\documentclass[12pt]{amsart}
\usepackage{graphicx}
\usepackage{amsmath}
\usepackage{subfigure}
\usepackage{amssymb} 
\usepackage[usenames,dvipsnames]{color}
\usepackage{pinlabel}
\usepackage{mathtools}
\usepackage{enumerate}
\usepackage{float}
\usepackage{tikz}
  \usetikzlibrary{hobby}
  \usetikzlibrary{patterns}

\newtheorem{thm}{Theorem}[section]  
\newtheorem{cor}[thm]{Corollary}

\newtheorem{defin}[thm]{Definition} 
\newtheorem{lemma}[thm]{Lemma} 
 
\newtheorem{prop}[thm]{Proposition} 
 
\newtheorem{ass}[thm]{Assumption} 
\newtheorem*{defin*}{Definition}

\newcommand{\bbb}{\mbox{$\beta$}}

\newcommand{\Ggg}{\mbox{$\Gamma$}}

\newcommand{\Db}{\mbox{$\mathcal{D}_B$}}
\newcommand{\Da}{\mbox{$\mathcal {D}_A$}}

\newcommand{\Sss}{\mbox{$\Sigma$}}

\newcommand{\bdd}{\mbox{$\partial$}}

\newcommand{\calA}{\mbox{$\mathcal A$}}

\newcommand{\fc}{\mbox{$\mathfrak c$}}

\newcommand{\Ebar}{\mbox{$\overline{E}$}}
\newcommand{\cbar}{\mbox{$\overline{c}$}}

\begin{document}  

\title{Uniqueness in Haken's Theorem}   

\author{Michael Freedman}
\address{\hskip-\parindent
        Microsoft Station Q\\ 
        University of California\\
        Santa Barbara, CA 93106-6105\\ 
        and\\
        Mathematics Department\\
        University of California\\
        Santa Barbara, CA 93106-3080 USA}
\email{michaelf@microsoft.com}

\author{Martin Scharlemann}
\address{\hskip-\parindent
        Martin Scharlemann\\
        Mathematics Department\\
        University of California\\
        Santa Barbara, CA 93106-3080 USA}
\email{mgscharl@math.ucsb.edu}

%\thanks{Research partially supported by National Science Foundation grants.}

\date{\today}

\begin{abstract}  Following Haken \cite{Ha} and Casson-Gordon \cite{CG}, it was shown in \cite{Sc} that given a reducing sphere or $\bdd$-reducing disk $E$ in a Heegaard split manifold $M$, the Heegaard surface $T$ can be isotoped so that it intersects $E$ in a single circle.  Here we show that when this is achieved by two different positionings of $T$, one can be moved to the other by a sequence of
\begin{itemize}
\item isotopies of $T$ rel $E$
\item pushing a stabilizing pair of $T$ through $E$ and
\item eyegelass twists of $T$.
\end{itemize}
This last move is inspired by one of Powell's proposed generators for the Goeritz group \cite{Po}. 
\end{abstract}

\maketitle

It is a classic theorem of Haken \cite{Ha} that any Heegaard splitting $M = A \cup_T B$ of a closed orientable reducible $3$-manifold $M$ is reducible; that is, there is an essential sphere in the manifold that intersects $T$ in a single circle.  Casson-Gordon \cite[Lemma 1.1]{CG} refined and generalized the theorem, showing that it applies also to essential disks, when $M$ has boundary.  More specifically, if $E$ is an essential disk (resp $2$-sphere) in $M$ then there is an essential disk (resp $2$-sphere) $E^*$, obtained from $E$ by ambient $1$-surgery and isotopy, so that $E^*$ intersects $T$ in a single circle. It is now known \cite{Sc} that in fact we may take $E^* = E$.  An alternative argument using sphere complexes appears in \cite{HS}, but it applies only in the case that each component of $E$ and of $\bdd M$ is a sphere. The broader case considered in \cite{Sc}, in which $E$ may contain disks and the boundary may contain higher genus surfaces, is crucial to the inductive step here.

More broadly, even if $M$ has spherical boundary components (so $A$ and $B$ may themselves be reducible) and $E$ is any properly embedded surface so that each component of $E$ is either a disk or a sphere (i. e. a {\em disk/sphere set}) it is shown in \cite{Sc} that $T$ can be isotoped so that it is {\em aligned} with $E$.  Roughly, `aligned' means that $T$ intersects each component of $E$ transversally in {\em at most} one circle.  In particular the disk/sphere set may contain components that lie entirely in $A$ or $B$.  

Here we consider a naturally related uniqueness question:  
Suppose disk/sphere sets $E_0$ and $E_1$ are both aligned with $T$ and the sets $E_0$, $E_1$  are isotopic rel $\bdd$ in $M$.   Is there such an isotopy $E_s \subset M, 0 \leq s \leq 1$ from $E_0$ to $E_1$ so that each $E_s$ is aligned with $T$?

Counterexamples spring to mind, even when $M$ is irreducible and each $E_i$ is simply a single disk. 
\bigskip

%\begin{example}
 {\bf Example:}  Suppose $E_0$ is a $\bdd$-reducing disk for $M$ which intersects $T$ in a single circle, so it is aligned with $T$.  Suppose $T$ is stabilized, with stabilizing disks $D_A \subset A$, $D_B \subset B$ and both disks are disjoint from $E_0$.  A regular neighborhood of $D_A \cup D_B$ is a ball $\bbb$ which $T$ intersects in a standard genus $1$ summand.  We call such a pair $(\bbb, T)$ a {\em standard bubble}.  We can imagine $\bbb$ as a small ball, the regular neighborhood of a point in the destabilized Heegaard surface $T'$.  Now isotope $\bbb$ along a path in $T'$ which passes once through $E_0$, and let $E_1$ be the result of pushing $E_0$ by the resulting ambient isotopy of $M$.  Then $E_0$ and $E_1$ are isotopic in $M$, but typically they can't be isotopic via disks aligned with $T$, since the circles $T \cap E_i, i = 0, 1$ are not isotopic in $T$.
%\end{example}
%  
%\begin{example}
\bigskip

{\bf Example:} More subtly, suppose  $D_A$, $D_B$ are disjoint essential disks in $A$ and $B$ respectively, and $\gamma$ is a path in $T$ connecting their boundaries.  The complex $D_A \cup \gamma \cup D_B$ is called an {\em eyeglass} for $T$ (\cite[Definition 2.1]{FS}).  Associated to such an eyeglass is an isotopy of $T$ in $M$ back to itself (with support near the eyeglass) called an {\em eyeglass twist}.  It is illustrated in Figure \ref{fig:eyeglass1}.  Suppose $E_0$ is aligned with $T$ and the circle $E_0 \cap T$ essentially intersects the bridge $\gamma$ of the eyeglass.  Then the disk $E_1$ obtained by pushing $E_0$ along by the resulting ambient isotopy of $M$ typically will not be isotopic to $E_0$ via aligned disks, again since the circles $T \cap E_i, i = 0, 1$ are not isotopic in $T$.
%\end{example}

 \begin{figure}[ht!]
\labellist
\small\hair 2pt
\pinlabel  $D_A$ at 110 130
\pinlabel  $\gamma$ at 155 133
\pinlabel  $D_B$ at 200 135
\endlabellist
    \centering
    \includegraphics[scale=0.6]{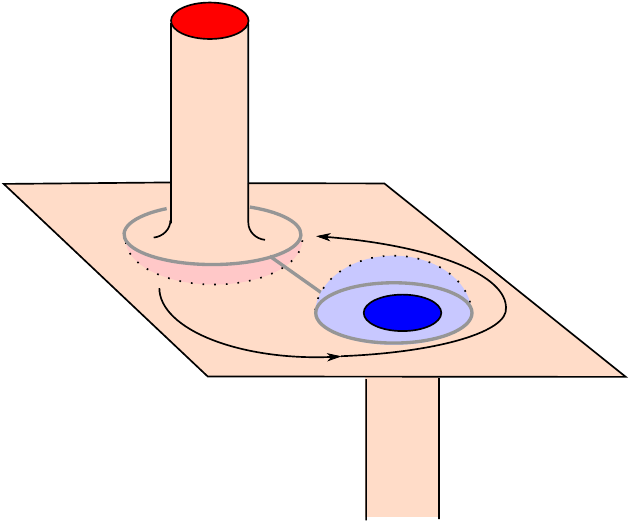}
    \caption{Eyeglass twist}  \label{fig:eyeglass1}
    \end{figure}

Our goal is to show that the two operations just described are essentially the only two obstacles to uniqueness.

\section{Background and results}

Suppose $M = A \cup_T B$ is a Heegaard splitting of a compact orientable $3$-manifold, in which (as in \cite{Sc}) $\bdd M$ may contain spheres.  In particular, since $\bdd_-A$ and $\bdd_-B$ may have spheres, the compression bodies $A$ and $B$ may themselves be reducible.  A sphere in $M$ is called inessential if it either bounds a ball in $M$ or is parallel in $M$ to a boundary component of $M$; a disk is inessential if it is parallel to a disk in $\bdd M$.

\begin{defin} \label{defin:reducer} 
A disk/sphere set $(E, \bdd E) \subset (M, \bdd M)$ is a compact properly embedded surface in $M$ so that each component of $E$  is either a disk or a sphere.  

%A disk component of $E$ that is essential and whose boundary lies in $\bdd_- A$ (resp $\bdd_- B$) is an $A$-disk (resp $B$-disk). 

A Heegaard surface $T$ and a disk/sphere set $E$ in $M$ are {\em aligned} if they are transverse, and each component of $E$ intersects $T$ in at most one circle. In addition, each disk component of $E - T$ is essential in either $A$ or $B$ and each annulus component of $E - T$ is vertical (see \cite[Section 2]{Sc} or below) in $A$ or $B$.  
\end{defin}

For example, a reducing sphere or $\bdd$-reducing disk for $T$, defined as a sphere or disk that intersects $T$ transversally in a single essential circle, are each important examples of an aligned disk/sphere set.  In our discussion there will be nothing lost by assuming that each component of a disk/sphere set is essential, so we implicitly make that assumption going forward.  
\bigskip

Suppose $(E, \bdd E) \subset (M, \bdd M)$ is an aligned disk/sphere set.  Let $S \subset M$ be a reducing sphere for $T$ that is disjoint from $E$ and cuts off a genus $1$ stabilizing summand of $T$ inside a ball that is disjoint from $E$.  Let $\gamma$ be an arc in $T$ with one end at a component $\Ebar$ of $E$, the other end at $S \cap T$, and $\gamma$ is otherwise disjoint from both $E$ and $S$.   Alter $\Ebar$  by tube summing it to $S$ along a neighborhood of $\gamma$ and call the result $\Ebar'$.  Replace $\Ebar$ by $\Ebar'$ in $E$ and call the result $E'$, a disk/sphere set still aligned with $T$.  

We can think of $S$ as a `bubble' that passes through $E$ to create $E'$, so we have:

\begin{defin}    $E'$ is obtained by a {\em bubble move} on $E$ along $\gamma$ with bubble $S$.
\end{defin}

  Note that $E$ and $E'$ are properly isotopic in $M$.
  \bigskip
  
  Now let $(D_A, \bdd D_A) \subset (A, \bdd A)$ and $(D_B, \bdd D_B) \subset (B, \bdd B)$ be disjoint essential disks that are also disjoint from $E$.  Let $\gamma$ be an arc in $T$ transverse to $E$, with one end at $\bdd D_A$, other end at $\bdd D_B$, and otherwise disjoint from $D_A \cup D_B$. Perform an eyeglass twist on $T$ using the eyeglass $D_A \cup \gamma \cup D_B$.  The eyeglass twist returns $T$ to itself, but may alter $E$.
  
\begin{defin} The image $E'$ of $E$ is said to be obtained from $E$ by an eyeglass twist.
\end{defin}

Note that $E'$ is still aligned with $T$, $E$ and $E'$ are properly isotopic in $M$ and, if $E$ is disjoint from the arc $\gamma$ then $E = E'$ .

\begin{defin} 
Suppose $E_0$ and $E_1$ are each disk/sphere sets aligned with $T$ in $M$.  An isotopy $E_s, 0 \leq s \leq 1$ from $E_0$ to $E_1$ in $M$ is an equivalence (and $E_0, E_1$ are equivalent) if $E_s$ is aligned with $T$ for all $s$.
\end{defin}

\begin{defin} 
$E_0$ and $E_1$ are {\em congruent} if a sequence of  equivalences, bubble moves and eyeglass twists carries $E_0$ to $E_1$.
\end{defin}

We intend to show:

\begin{thm}  \label{thm:main}  If $E_0, E_1$ are disk/sphere sets that are properly isotopic in $M$ and are both aligned with $T$, then $E_0$ and $ E_1$ are congruent.
\end{thm}

In conjunction with \cite{Sc}, this means that any disk/sphere set in $M$ is isotopic in $M$ to a set aligned with $T$ that is unique up to congruence.  

%\bigskip
%
%{\bf Example:}  Theorem \ref{thm:main} is obvious for reducing spheres $E_0, E_1$ that intersect $T$ in disjoint circles:  Then each component of $E_0 \cap E_1$ is a circle lying in either $A$ or $B$.  An innermost one in $E_0 \cap A$, say, cuts off also a subdisk of $E_1 \cap A$.  Since $A$ is irreducible, the latter disk can be isotoped to the former in $A$. Eventually such isotopies make $E_0$ and $E_1$ disjoint, so they are parallel in $M$.  The splitting that $T$ induces inside of the collar between them is simply a sum of stabilizing pairs, per Waldhausen (\cite{Wa}, \cite{R}), which can be passed through $E_0$ bubble by bubble until the spheres are equivalent. A similar argument applies if $E_0, E_1$ are $\bdd$-reducing disks.  So the interest focuses on cases in which $E_0 \cap E_1 \cap T \neq \emptyset$.  

\section{Sweepouts, spines, and labels for the graphic}

Here we briefly review the classical sweep-out technology on $M = A \cup_T B$.  Here $A$ (and similarly $B$) is a compression-body, with $\bdd_+ A = T$ and $\bdd_-A = (\bdd A) - T = \bdd M \cap A$.  This means that $A$ can be viewed (dually to the original definition in \cite{Bo}) as a compact connected orientable $3$-manifold obtained from $\bdd_- A \times I$ (which may be disconnected) by attaching $1$-handles to $\bdd_- A \times \{1\}$.  %The boundary of $A$ is the disjoint union of $\bdd_- A =$ surface $\times \{0\}$ and a connected surface denoted $\bdd_+ A$.
From this construction we see that $A$ deformation retracts to the union of $\bdd_- A$ and the cores of the $1$-handles, where the latter are extended down through $\bdd_- A \times I$ via the product structure.  

More generally, a {\em spine} $\Sss$ of a compression body $C$ is the union of $\bdd_- C$ and a certain type of graph in $C$: all valence $1$ vertices in the graph lie on $\bdd_- C$; all other vertices have valence $3$; and $C$ deformation retracts to $\Sss$; indeed $C - \Sss \cong \bdd_+ C \times (0, 1]$. (Sometimes we will not distinguish between $\Sss$ and a thin regular neighborhood of $\Sss$.)  $C$ has many spines, but in an argument that goes back to Whitehead \cite{Wh} (who was concerned with spheres, not disks) one can change one spine (viewed topologically) to any other by a sequence of ``edge-slides'', in which one edge is slid over others and along $\bdd_- C$ \cite[Section 1]{ST}, \cite[Proposition 3.4]{Sc}.  
\bigskip

A properly embedded annulus in a compression body $C$ is {\em spanning} if its two boundary components lie, one each, in $\bdd_+ C$ and $\bdd_- C$.  A disjoint collection $\calA$ of spanning annuli is {\em vertical} if there is a complete collection of meridian disks $\Delta$ for $C$ that is disjoint from $\calA$ and $\calA$ is vertical in $C - \eta(\Delta) \cong \bdd_- C \times I$. (It suffices that $\Delta$ be disjoint from a vertical spanning arc in each annulus \cite[Proposition 2.5]{Sc}).  

Let $E \subset C$ be a disjoint collection of vertical annuli, essential disks and essential spheres.   Essentially the same argument as in \cite[Section 1]{ST} shows that there is a spine $\Sss$ for $C$ with the properties:
\medskip

\begin{enumerate}[a)]
\item Each annulus in $E$ intersects $\Sss$ only in $\bdd_- C$.
\item Each disk in $E$ whose boundary lies on $\bdd_+ C$ intersects $\Sss$ in a single point in an edge.
\item Each sphere in $E$ intersects $\Sss$ in a single point in an edge.
\item Each disk in $E$ whose boundary lies on $\bdd_- C$ intersects $\Sss$ only in $\bdd_- C$.
\end{enumerate}
See Figure \ref{fig:comports}.

 \begin{figure}[ht!]
\labellist
\small\hair 2pt
\pinlabel  $b)$ at 120 240
\pinlabel  $c)$ at 160 130
\pinlabel  $d)$ at 230 120
\pinlabel  $a)$ at 370 135
\endlabellist
    \centering
    \includegraphics[scale=0.6]{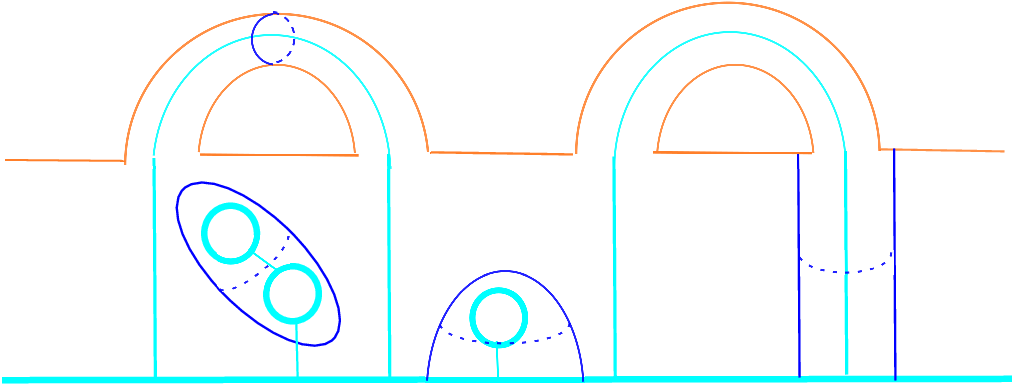}
    \caption{Comporting spine}  \label{fig:comports}
    \end{figure}
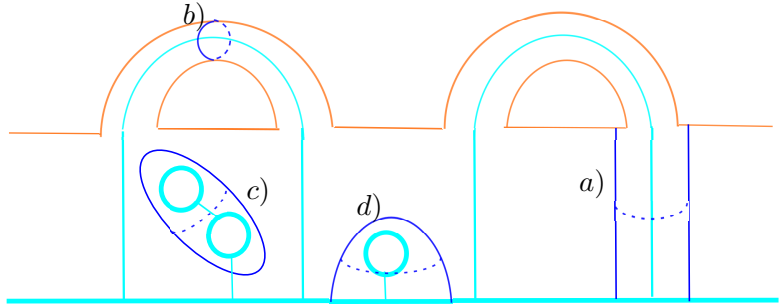

Notice that in the first two cases the complement of $\Sss$ in the component of $E$ is a half-open annulus; in the second two it is an open disk.  One can choose the parameterization $C - \Sss \cong \bdd_+ C \times (0, 1]$ so that the half-open annuli components of $E - \Sss$ are parameterized as $(E \cap \bdd_+ C) \times (0, 1]$ and the open disk components are parameterized by the standard height function on the interior of a round hemisphere of radius $<1$ in upper half-space.   We will say that such a spine and parameterization {\em comports with $E$}.  Note that, via Hatcher's work \cite{Ha1}, \cite{Ha2}, the exact parameterization involves no choice, in the sense that its space of parameters is contractible.

Combining these ideas, if $E' \subset C$ is another such collection, then one can move from a spine (and associated parameterization) that comports with $E$ to one that comports with $E'$ via a sequence of edge slides.  
\bigskip

Now we export all these ideas to the setting at hand:  a Heegaard split $M = A \cup_T B$ and two aligned disk/sphere sets $E_0$ and $E_1$ that are isotopic in $M$.  The two compression bodies will not be treated the same in the argument, so begin by assigning the names $A$ vs $B$ to the two compression bodies to ensure the following: If any components of the aligned $E_i$ are disjoint from $T$ then there is such a component in $A$.  

Each $E_i, i = 0, 1,$ intersects each compression-body $A$ (resp $B$) in a collection of vertical annuli, essential spheres and essential disks $E_{i, A} = E_i \cap A$ (resp $E_{i, B} = E_i \cap B$).  Choose spines $\Sss_{i, A} \subset A$ (resp $ \Sss_{i, B}\subset B$) so that each comports with $E_{i, A}$ (resp $E_{i, B}$).  For each $i = 0, 1$ combine the comporting parameterizations in each compression-body, to parameterize the entire complement of the spines in $M$ as $T \times (0, 1)$, picking the convention that the spine of $A$ is the limit of $T \times \{t\}$ as $t \to 0$. Then the complement of the spines in $M$ is swept-out by copies of $T$ in such a way that, generically, each copy of $T$ intersects each component of $E_i$ in at most one circle. Denote the copy $T \times \{t\}$ in such a sweepout by $T_t$.

The core argument will mirror that of \cite[Section 4]{FS}, with the isotopy $E_s, 0 \leq s \leq 1$ from $E_0$ to $E_1$ replacing what was there a sweepout of $S^3$ by level $2$-spheres.  In addition we use $s$ to simultaneously parameterize a movie of the sequence of edge slides on the spines that take $\Sss_{0, A} \cup \Sss_{0, B}$ to $\Sss_{1, A} \cup \Sss_{1, B}$.  
Together, this sweep-out and the isotopy $E_s$ (together with edge slides on the spines) produce a ``graphic" $\Ggg$ in the $(t, s)$-square $I \times I$.  

The graphic consists of open regions where $E_s$ and $T_t$ intersect transversely, edges or ``walls" where the two have a tangency, and cusp points where two types of tangencies cancel. As argued in \cite{RS} and discussed a bit further below, only domain walls corresponding to saddle tangencies need to be tracked. Cusps and tangencies of index 2 or 0 can be erased as they amount only to births/deaths of inessential simple closed curves of intersection in $E_s \cap T_t$. The most interesting event which occurs are transverse crossings of saddle walls; at this point two independent saddle tangencies occur. 

In more detail, consider first a region $R$ in which $E = E_s$ and $T = T_t$ intersect transversally, in a set $\fc$ of simple closed curves.  A curve $c \in \fc$ is either essential in $T$ or bounds a disk $D_c \subset T$.  In the latter case, the union of $D_c$ with a disk $E_c$ that $c$ bounds in a component $\Ebar$ of $E$ define an immersed sphere $S_c$ in $M$, which may or may not be null-homotopic in $M$.  If $\Ebar$ is a sphere, so $c$ bounds two disjoint disks in $\Ebar$, denote the corresponding immersed spheres $S_c$ and $S'_c$.  Note that $S_c$ and $S'_c$ cannot both be null-homotopic, since we are assuming $\Ebar$ is essential in $M$.  

\begin{defin} The curve $c \in \fc$ is an {\em essential} circle of intersection if either
\begin{itemize}
\item $c$ is an essential circle in $T$
\item $c$ is inessential in $T$, $\Ebar$ is a disk, and the immersed sphere $S_c$ is not null-homotopic
\item $c$ is inessential in $T$, $\Ebar$ is a sphere, and neither $S_c$ nor $S'_c$ is null-homotopic
\end{itemize}
Otherwise, $c$ is an {\em inessential} circle of intersection.  
\end{defin}

Denote by $\fc_e$ the collection of essential curves in $\fc$.  

\begin{lemma}  \label{lemma:hatcher} Suppose a curve $c \in \fc_e$ bounds a disk $E_c \subset E$ so that
each component of $\fc \cap int(E_c)$ is inessential. 

Then there is an isotopy of $E_c$ rel $c = \bdd E_c$ into either $A$ or $B$; the resulting disk in $A$ (resp $B$) is well-defined up to isotopy rel $\bdd$ in $A$ (resp $B$).  
\end{lemma}

Call this resulting properly embedded disk in $A$ or $B$ the {\em resolve} of $E_c$.  

\begin{proof} If $int(E_c)$ contains no circles in $\fc$ then there is nothing to prove.  If it does contain such circles, the argument is a minor variant of the standard innermost disk argument:  consider an innermost circle $\cbar$ bounding a disk $E_{\cbar} \subset E_c$.  Since $\cbar$ is innermost, $E_{\cbar}$ is disjoint from $T$.  Since $\cbar$ is inessential, it bounds a disk $D_{\cbar} \subset T$ and the embedded sphere $D_{\cbar} \cup_{\cbar} E_{\cbar}$ is null-homotopic in $M$.  Then this sphere bounds a ball and, by construction, that ball lies in either $A$ or $B$, say $A$.  The ball can be used to isotope $E_{\cbar}$ past $D_{\cbar}$, perhaps carrying other components of $E \cap A$ with it.  The isotopy removes all curves of intersection that lie in $int(D_{\cbar})$ (`secondary curves') as well as $\cbar$ itself.  The isotopy is also a homotopy, so whether each remaining curve in $\fc$ is essential or inessential is unchanged. 

We are indebted to Allen Hatcher for the following proof of uniqueness up to isotopy in $A$.  The argument is quite analogous to those used in his groundbreaking work \cite{Ha1}; an interested reader is particularly encouraged to examine the discussion in and around  \cite[344-345]{Ha1}.  

The process described above depends on the order in which innermost circles are chosen at each successive stage; we will show that different choices lead to disks in $A$, say, that are isotopic in $A$.  The order can be specified by a function $f$ which assigns a number $f(\cbar)$ in $(0,1)$ to each circle $\cbar \in (\fc \cap int({E_c}))$ subject to the ordering condition that $f(\cbar)<f(\cbar')$ if $\cbar$ lies inside $\cbar'$ in $E_c$.  
The isotopy of $E$ described above (which in Figure \ref{fig:Hatcher} we view as parameterized by a new variable $u \in [0, 1]$) eliminates $\cbar$ and its secondary curves during the time interval $[f(\cbar),f(\cbar)+\epsilon]$. 

 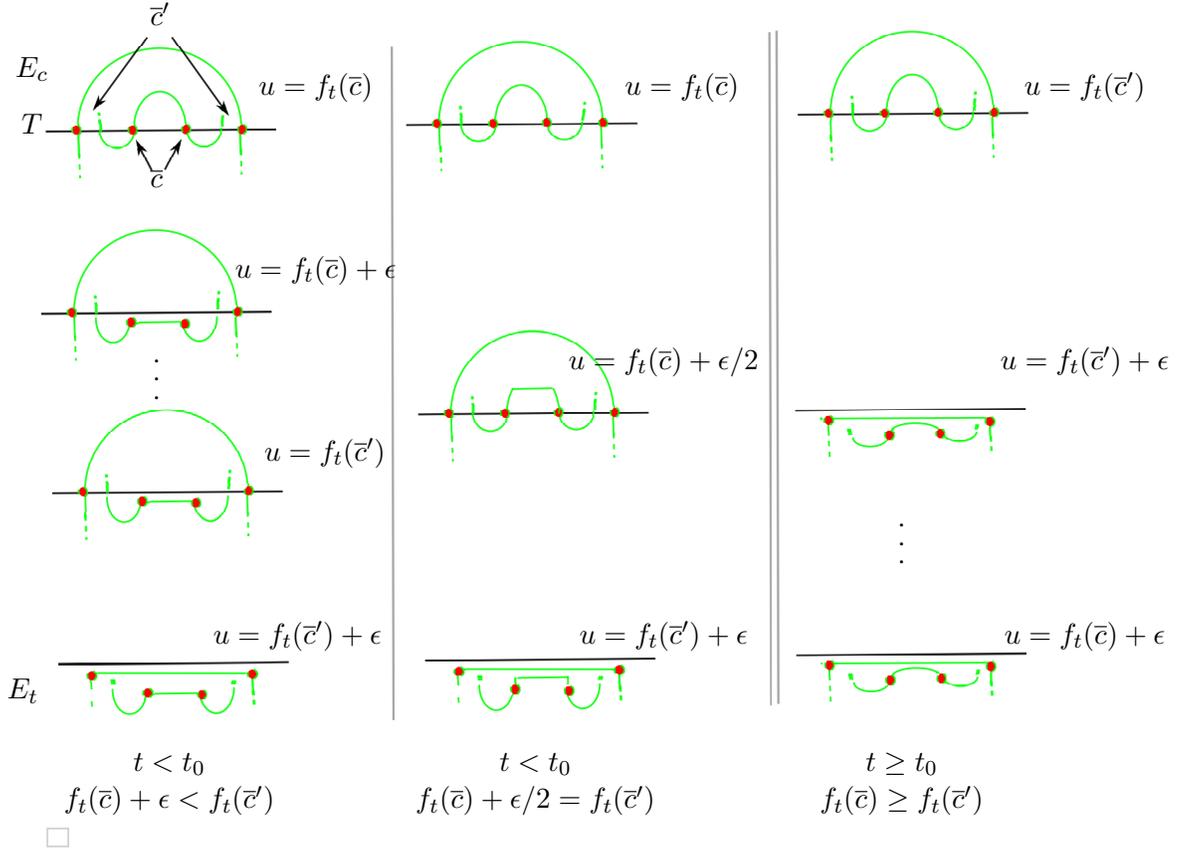
\begin{figure}[ht!]
\labellist
\small\hair 2pt
\pinlabel  $\cbar'$ at 95 490
\pinlabel  $\cbar$ at 93 400
\pinlabel  $E_c$ at 25 460
\pinlabel  $T$ at 25 430
\pinlabel  $u=f_t(\cbar)$ at 180 450
\pinlabel  $u=f_t(\cbar)+\epsilon$ at 180 350
\pinlabel $\cdot$ at 93 300
\pinlabel $\cdot$ at 93 290
\pinlabel $\cdot$ at 93 280
\pinlabel  $u=f_t(\cbar')$ at 185 250
\pinlabel  $u=f_t(\cbar')+\epsilon$ at 170 150
\pinlabel  $t<t_0$ at 100 80
\pinlabel $f_t(\cbar)+\epsilon<f_t(\cbar')$ at 100 60
\pinlabel  $E_t$ at 20 120
%next column
\pinlabel  $u=f_t(\cbar)$ at 380 450
\pinlabel  $u=f_t(\cbar)+\epsilon/2$ at 370 300
\pinlabel  $u=f_t(\cbar')+\epsilon$ at 370 150
\pinlabel  $t<t_0$ at 300 80
\pinlabel $f_t(\cbar)+\epsilon/2=f_t(\cbar')$ at 300 60
%next column
\pinlabel  $u=f_t(\cbar')$ at 600 450
\pinlabel  $u=f_t(\cbar')+\epsilon$ at 600 300
\pinlabel $\cdot$ at 500 200
\pinlabel $\cdot$ at 500 210
\pinlabel $\cdot$ at 500 190
\pinlabel  $u=f_t(\cbar)+\epsilon$ at 600 150
\pinlabel  $t\geq t_0$ at 500 80
\pinlabel $f_t(\cbar)\geq f_t(\cbar')$ at 500 60
\endlabellist
    \centering
    \includegraphics[scale=0.7]{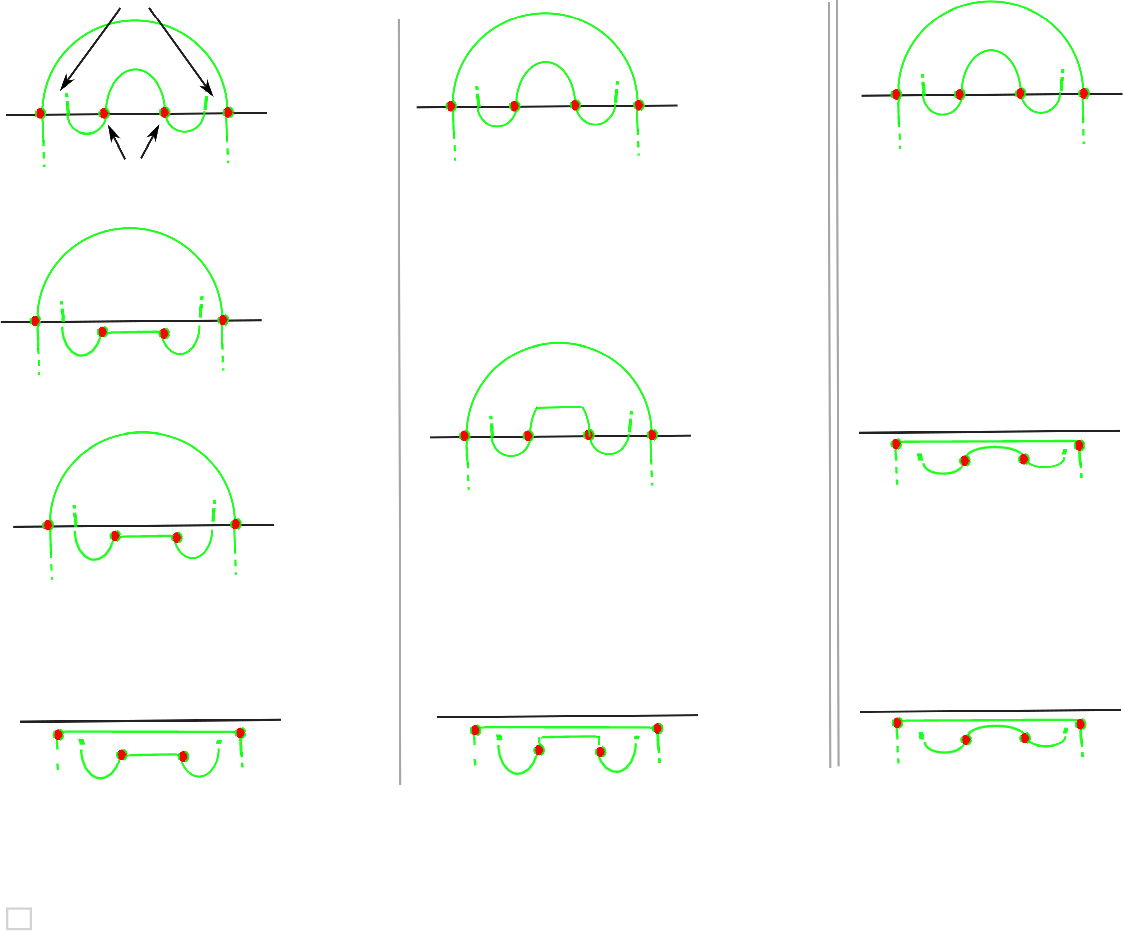}
    \caption{$t$ parameterizes isotopy of $u$-isotopies}  \label{fig:Hatcher}
    \end{figure}

Suppose one does this in two different ways, using functions $ f_0$ and $f_1$ to give isotopies of $E_c$ to disks $E_0$ and $E_1$ in $A$.  Since $f_0$ and $f_1$ are both maps of the (finite) set $\fc \cap int({E_c})$ to the convex set $(0, 1)$, we can connect $f_0$ and $f_1$ by a homotopy $f_t=(1-t)f_0 + tf_1$, with each $f_t$ satisfying the ordering condition.  In general position each $f_t$ will be injective, with finitely many exceptions, namely where two values are interchanged for circles $\cbar$ and $\cbar'$, neither of which lies inside the other in $E_c$.  If the balls $B$ and $B'$ determined by $\cbar$ and $\cbar'$ are disjoint the isotopies eliminating $\cbar$ and $\cbar'$ by pushing across $B$ and $B'$ have disjoint supports so can be performed independently.     If $B$ and $B'$ intersect there are two possibilities according to whether one is contained in the other (ie they lie on the same side of $T$) or they intersect just in a disk in their boundaries (they lie on opposite sides of $T$).  The two cases are treated in the same way. 

Suppose that $\cbar$ lies inside $\cbar'$ in the disk in $T$ bounded by  $\cbar'$ and suppose that $f_t(\cbar)=f_t(\cbar')$ for $ t=t_0$, with $f_t(\cbar)<f_t(\cbar')$ for $t<t_0$ and$f_t(\cbar)>f_t(\cbar')$ for $t>t_0$.  
Then $\cbar$ is primary for $t<t_0$ and secondary for $t>t_0$.  For $t<t_0$ we would perform the isotopy eliminating $\cbar$ during the time interval $[f_t(\cbar),f_t(\cbar)+\epsilon]$.  
However we modify this prescription by truncating this isotopy so that we only perform the initial part of it lying in the time interval $[f_t(\cbar),f_t(\cbar')]$.  Thus for $t=t_0$ we do not perform any of the isotopy.  
Then in the interval $[f_t(\cbar'),f_t(\cbar')+\epsilon]$ we do the isotopy eliminating $\cbar'$.  
For $t>t_0$ we only do the isotopy eliminating $\cbar'$, which automatically eliminates $\cbar$ as well.

In the end, each $f_t$ determines a process to isotope $E_c$ to a disk $E_t \subset A$ and the construction ensures that $E_t$ varies continuously with $t$.  So the disks $E_t$ describe an isotopy from $E_0$ to $E_1$ that lies entirely in $A$.
\end{proof}

Following Lemma \ref{lemma:hatcher}, label the region $R$  of the graphic as follows:
\begin{itemize}
\item Ignore inessential circles in $\fc$. 
\item Label the region $A$ if there is a component of $E$ that lies entirely in $A$ or if there is a circle $a \in \fc_e$ so that either 
\begin {itemize} 
\item  $a$ is innermost in $E$ among essential circles in $\fc$, and the resolve of the disk it bounds lies in $A$ or
\item  $a$ is parallel in $E$ to a boundary component lying in $\bdd_-A$, and the collar between them contains only inessential circles in $\fc$.% that are also inessential in the collar.
\end{itemize}
\item Label a region $B$ if there is neither a circle $a$ as above nor a component of $E$ entirely in $A$ but either
\begin{itemize} 
\item there is a component of $E$ that lies entirely in $B$ or
\item there is at least one circle $b \in \fc_e$ that is innermost in $E$ among circles of $\fc_e$, and the resolve of the disk it bounds lies in $B$.  
\end{itemize}
\end{itemize}
%
%Label a region of the graphic as follows:
%\begin{itemize}
%\item Ignore circles in $T_t \cap E_s$ that are inessential in $T_t$, 
%\item Label the region $A$ if there is a circle $a$ of $T_t \cap E_s$ so that either 
%\begin {itemize} 
%\item  $a$ is innermost in $E_s$ among essential circles in $T_t \cap E_s$, and the disk in $E_s$ that it bounds lies in $A$ or
%\item  $a$ is parallel in $E_s$ to a boundary component, and the collar between them lies in $A$
%\end{itemize}
%\item Label a region $B$ if there are {\em no} such circles $a$ as above, but there is {\em at least one  circle} $T_t \cap E_s$ that is innermost in $E_s$ among essential circles in $T_t \cap E_s$  and the disk in $E_s$ that it bounds lies in $B$. 
%\end{itemize}

Note that the definition of the labeling breaks symmetry: A vertical annulus in a collar of $\bdd A$ is counted as if it were an innermost disk but a vertical annulus in a collar of $\bdd B$ is not; and a region in which there are components of $E$ or resolves in both $A$ and $B$ is labeled $A$.  For example, if some component of $E$ is an aligned sphere intersecting $T$ in a single circle, then the region is labeled $A$.  

\bigskip 

In the figures illustrating our argument we will distinguish between the compression bodies $A$ and $B$ by color: pinkish (nominally red) will denote $A$, while azure (nominally blue) will denote $B$.  This distinction will color regions of $E$ cut out by $T$ alternately red and blue.  

For example, consider Figure \ref{fig:hidden}. The bi-colored horizontal plane shows part of a component of $E$.  Two parts of $T$ are shown
\begin{itemize}
\item a large-diameter vertical annulus, separating the visible part of $E$ into a blue unbounded region and a red `pair of pants'; and 
\item an inverted-U-shaped annulus that separates a blue $1$-handle of $B$ (the `blue tube')  from the part of $A$ that contains the red pair of pants.  
\end{itemize}
Figure \ref{fig:hidden} shows the blue tube bounded by part of $T$ being lowered through $E$.  In so doing the $(s, t)$ parameter designating the two surfaces passes from one region of the graphic to another.  An astute viewer will notice a gray area at the top of the blue tube reflecting ambiguity on what might lie there:  Is the top of the tube just a disk, or does a chimney filled with blue ascend through it?  This could be an important question, as we will discuss shortly.  

We will also use the red-blue coloring scheme on the graphic: Regions that are labeled $A$ will be colored red; those labeled $B$ will be colored blue.  Skip ahead to Figure \ref{fig:graphicsketch} to see how the coloring scheme might then appear in $I \times I$ containing the graphic.  The idea of the proof of Theorem \ref{thm:main} can already be discerned in this figure: Ultimately we will walk around the outside boundary of the big red region and observe that every step corresponds to some combination of an isotopy, a bubble move or an eyeglass twist.  

In more detail, the focus will be on the right-hand side of the region, the segment of the boundary of the region that lies between the points $(t, s) = (1, 0)$ and $(1, 1)$.  The labeling ensures that these points respectively represent the aligned system $E_0$ and $E_1$ and furthermore that any point in this segment for which $t = 1$ corresponds to some alignment, one that does not change with $s$ as long as $t =1$. This is discussed in Section \ref{sect:graphicbound}.

The remaining edges in the segment, those that lie in the interior of the square, and so border both a red and a blue region, are the central object of study.  That they are adjacent both to red and blue regions ensures that they provide a `weak alignment' of the system $E$ (analogous to weak reduction in Heegaard theory).  Section \ref{sect:weakreduce} discusses a natural algorithm, that proceeds from such a weak alignment to a full alignment of $E$, an alignment that is, via arguments in Sections \ref{sect:weakreduce} and \ref{sect:appendage}, well-defined up to congruence.  The critical point is that the weak alignment provides a way of breaking $M$ up into Heegaard split submanifolds, each with a lower genus splitting, where we can apply an inductive assumption.

The remaining crucial issue is whether the algorithmically defined alignment of $E$ changes when passing through a vertex in the graphic; that is, whether adjacent `border edges' (those edges in the graphic dividing red from blue) determine the same alignment.  Sections \ref{sect:forbidden} and \ref{sect:graphicreturn} show that indeed they do.  This completes the argument: there is no change in alignment of $E$ (up to congruence) in moving along the segment from the corner of the graphic representing $E_0$ to the corner representing $E_1$.  

\bigskip

Return now to the ongoing argument:  The first labeling rule above -- ignore circles of intersection that are inessential in $T_t$ -- raises a  {\it caveat}:

When we say that an essential circle $a$ bounds a disk in $A$ (similarly for $B$),  what is technically meant is that there is a planar surface $P \subset A$ whose boundary consists of $a$ and, possibly, a collection of circles that are inessential in $\fc$.  

\begin{figure}[ht!]
\labellist
\small\hair 2pt
\pinlabel  $E$ at 30 120
%\pinlabel  $\gamma$ at 155 133
\pinlabel  $T$ at 195 150
\pinlabel  $T$ at 185 100
\endlabellist
    \centering
    \includegraphics[scale=0.6]{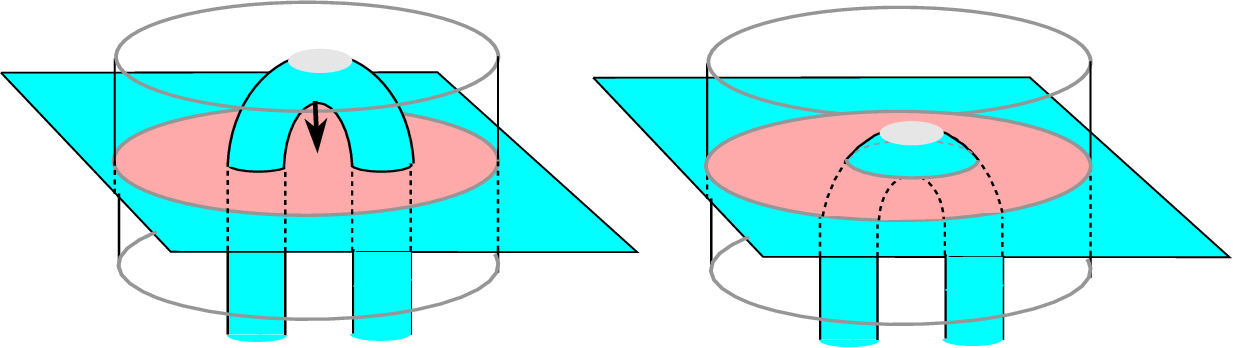}
\caption{Label change or not?} 
 \label{fig:hidden}
    \end{figure}

As a consequence, crossing from one region of the graphic to another may change the label from $B$ to $A$ in surprising ways, as shown in Figure \ref{fig:hidden}.  As the saddle point in $T$ passes down through $E$, the label will change from $B$ to $A$ if the grey disk at the top is inessential in $B$.  If the grey disk is essential in $B$, so $T$ ascends beyond it, the label remains $B$.  The difference between the two situations cannot be determined just by examining the behavior of $E \cap T$.  

\section{Labels around the boundary of the graphic} \label{sect:graphicbound}

In thinking about the labeling scheme, consider first the situation near $s = 0, 1$, where the parameterization $T \times (0, 1)$ in each case comports with $E$.  Observe first that the spine of $A$ must intersect some component of $E$: indeed a component disjoint from the spine can be taken to lie entirely in $B$, but our choice of  names $A$ vs $B$ then guarantees that there is a component of $E$ that can be made to lie entirely in $A$ and therefore must intersect the  spine of $A$.  Such a component in $E_i$ is swept out by a single circle (which will disappear entirely,  at a maximum, on any component that lies entirely in $A$).  As $T_t$ rises, one side of the circle is a disk or annulus lying entirely in $A$; if the circle disappears the entire component of $E_i$ lies in $A$ .  Thus all regions near $s = 0, 1$ are labeled $A$.  

Near $t = 0$, $T_t$ is near a spine of $A$, which as we have seen must intersect some component of $E_s$.  The intersection with the spine is either a point, or possibly a component of $\bdd E_s$ in $\bdd_- A$.  This means that near $t=0$, $S_s \cap T_t$ will cut out from $E_s$ a small disk in $A$ (or a thin annulus near a component of $\bdd E_s$ at $\bdd_- A$).  Thus the regions adjacent to $t=0$ are again all labeled $A$.  

The labeling of regions near $t=1$ is more subtle and contains a warm-up for the general case.   Because $T$ is near a spine of $B$, each circle of intersection with $E_s$ either bounds a disk in $B$ or is parallel in $B$ to a boundary component, as we have just noted.  Consider first the three simple cases that can arise when $E = E_s$ is a single component:
\begin{itemize}
\item Suppose $E$ is a reducing sphere for $M$.  Since each component of $E \cap T$ bounds a small disk in $B$, a region will be labeled $A$ if and only if there is at most one circle of intersection, in other words, if and only if $E$ is aligned with $T$.
\item If $E$ is a $\bdd$-reducing disk whose boundary lies in $\bdd_- B$ then $T$ intersects $E$ in at least a circle parallel to $\bdd E$ and the collar lies in $B$.  But if there is any other circle of intersection, the small disk it bounds lies in $B$, so there are no disks in $A$.  Again, a region will be labeled $A$ if and only if $E$ is aligned with $T$, indeed a $\bdd$-reducing disk for $T$.
\item Suppose $E$ is a $\bdd$-reducing disk for $M$ whose boundary lies on $\bdd_- A$.  Then $E$ can only intersect the spine of $B$ in points, and if it does so in more than one point then $E \cap A$ would not contain a disk or annulus and so would not be labeled $A$.  We deduce that the region is labeled $A$ if and only if the spine of $B$ intersects $E$ in at most one point, in which case $T$ intersects $E$ in at most one circle.  So, once again, the region is labeled $A$ if and only if $E$ is aligned with $T$.
\end{itemize}

Hence we have %(see Figure \ref{fig:graphicsketch}):

\begin{prop} Suppose $E$ has a single component and the regions adjacent to the side $t = 1$ are all labeled $A$.  Then $E_0$ and $E_1$ are equivalent aligned disks.  \qed
\end{prop}

When $E$ has many components, the argument is more complicated, since our labeling scheme assigns the label $A$ if just one component of $E_s$ is aligned.  So at this point we will make a crucial inductive assumption.  Following \cite{Sc} define the {\em complexity} of $M$ to be the pair $(g, n)$, lexicographically ordered, where $g$ is the genus of $T$ and $n$ is the number of spheres in the boundary of $M$.  

\begin{ass} \label{ass:induct}  Theorem \ref{thm:main} is true for any $(M', T')$ of complexity lower than $(M, T)$.  \end{ass}

Note that Theorem \ref{thm:main} is obvious when $genus(T) = 1$ and $n = 0$.  
\bigskip

Suppose in a region of the graphic a component $\overline{E}$ of  $E = E_s$ is aligned.  Then there is a natural way of isotoping $T$ rel $\overline{E}$ to align all of $E$, as described in \cite[Proposition 4.2]{Sc}:  Cut along $\overline{E}$ to obtain a new Heegaard split manifold $M' = A' \cup_{T'} B'$ (disconnected if $E$ is separating).  Each component of $(M', T')$ has complexity less than $(M, T)$ so not only can $T'$ be aligned in $M'$ with the family $E - \overline{E}$ (\cite{Sc}) but by Assumption \ref{ass:induct} this alignment for $(M', T')$ is well-defined up to congruence.  

Such an alignment also aligns $T$ with all of $E$, though there is one technical point: When $\overline{E}$ is not disjoint from $T$, so it intersects $T$ in a single circle $c$, we must ensure that the isotopy used to align $T'$ does not create more circles of intersection parallel to $c$ in $T$.  For this we  use \cite[Proposition 3.8]{Sc} to ensure that the tube $D^2 \times I$ in $B$, say, bounded by $c \times I$ (whose core $\{0\} \times I$ is denoted there by arcs $\beta$) is properly isotoped so as not to intersect other components of $E$.  Once this is done, it follows as in \cite[Proposition 3.7]{Sc} that these arcs $\beta$ may be properly isotoped at the appropriate moment to be disjoint from any bubble that is about to pass through $E$ or from the lenses involved in any eyeglass twist disjoint from $E$.  The upshot is that the alignment of $T$ with $E$ just described is well-defined up to congruence.  We say that this alignment is generated by the initial alignment of $\overline{E}$.  

\begin{lemma} \label{lemma:generate}
Suppose $\overline{E}'$ is another aligned component of $E$.  Then the alignment generated by $\overline{E}'$ is congruent to that generated by $\overline{E}$.

If $E$ is aligned, then the alignment is congruent to that generated by any of its members.

\end{lemma}

\begin{proof} The alignment obtained as above by cutting along both $\overline{E}'$ and $\overline{E}$ is congruent to that generated by either, per Assumption \ref{ass:induct}.
\end{proof}

In view of Lemma \ref{lemma:generate} we can simply call such an alignment in the region {\em internally generated} without naming the component of $E$ that generates it.  

\begin{lemma} \label{lemma:pairgenerate}
Suppose two regions of the graphic, adjacent along an edge of each, have internally generated alignments.  Then these alignments are congruent.  
\end{lemma}

\begin{proof} Let $\overline{E}$ and $\overline{E}'$ be generators in adjacent regions $R, R'$ respectively. If $\overline{E} = \overline{E}'$ congruence follows by definition, so we assume $\overline{E} \neq \overline{E}'$.  The edge between regions $R$ and $R'$ represents $E$ passing through a single saddle tangency with $T$, a point that may lie on $\overline{E}$ or $ \overline{E}'$ but not both.  Thus at least one of the two is a generator in both regions, from which congruence of the alignments follows by Lemma \ref{lemma:generate}.  
\end{proof}

Return now to the setting for Theorem \ref{thm:main} and we have:

\begin{prop} \label{prop:side1} Suppose the regions adjacent to the side $t = 1$ are all labeled $A$.  Then $E_0$ and $E_1$ have congruent alignments.
\end{prop}
\begin{proof} The label $A$ implies that each region adjacent to the side $t=1$ has a self-generated alignment.   Lemma \ref{lemma:pairgenerate} ensures that the congruence class of the alignment doesn't change as we move along the side $t=1$ from $E_0$ to $E_1$.
\end{proof}

\section{A forbidden labeling around a vertex} \label{sect:forbidden}

Focus now on how labels behave around a vertex in the interior of the graphic $\Gamma$. Such a vertex corresponds to a position of $T = T_t$ in which it has two simultaneous tangency points with $E = E_s$. The non-trivial cases arise when both points of tangency lie on a single component $\Ebar \subset E$. If $\Ebar$ is a disk, a simple combinatorial argument shows that there are 15 possible configurations of these tangency points, shown in Figure \ref{fig:vertexresolved}.   The same diagram can be used when $\Ebar$ is a sphere, but far fewer panels are needed because of the extra symmetry this introduces.  For example, panels 10, 11 and 12 are the same in a sphere, as are 13 and 14. We will proceed assuming $\Ebar$ is a disk; if it is a sphere, just delete an open disk near a point in $A$, converting it to an $A$-disk, and apply the arguments there.

 There are typically many more circles in $\Ebar \cap T$ than are shown in the panels of Figure \ref{fig:vertexresolved}; these only show the components containing tangency points.  The two tangency points will be denoted $\rho = \rho_{\pm}$; the 4 quadrants near it correspond to the 4 ways of resolving the tangencies, each by perturbing $T$ slightly near $\rho$.

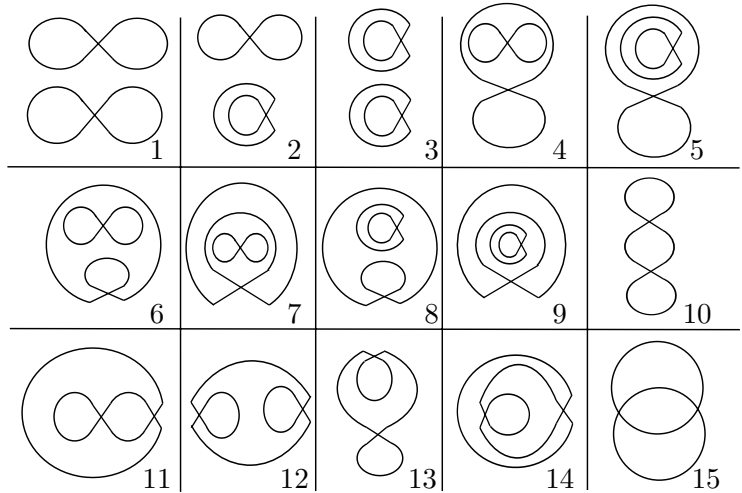
\begin{figure}[ht!]
\labellist
\small\hair 2pt
\pinlabel  $1$ at 85 195
\pinlabel  $2$ at 160 195
\pinlabel  $3$ at 235 195
\pinlabel  $4$ at 305 195
\pinlabel  $5$ at 380 195
\pinlabel  $6$ at 85 105
\pinlabel  $7$ at 160 105
\pinlabel  $8$ at 235 105
\pinlabel  $9$ at 305 105
\pinlabel  $10$ at 380 105
\pinlabel  $11$ at 85 15
\pinlabel  $12$ at 160 15
\pinlabel  $13$ at 230 15
\pinlabel  $14$ at 305 15
\pinlabel  $15$ at 385 15
\endlabellist
    \centering
    \includegraphics[scale=0.7]{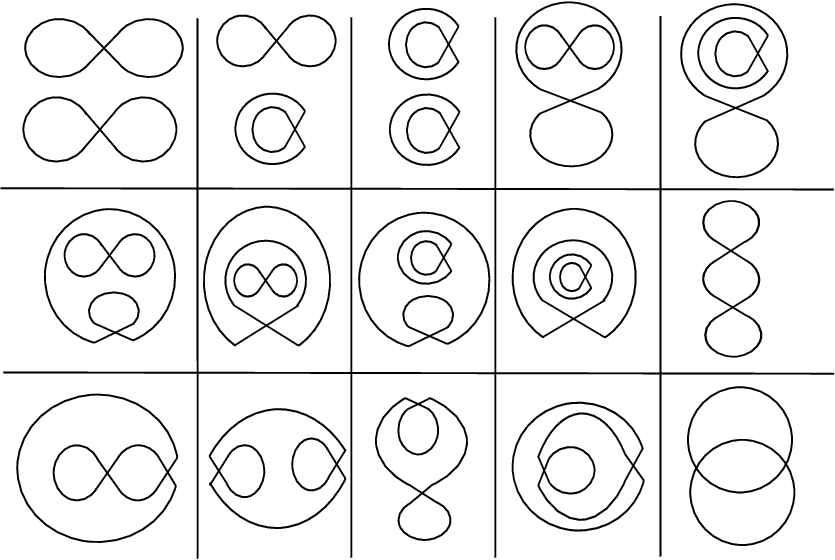}
\caption{At a vertex in the graphic $\Gamma$ } 
 \label{fig:vertexresolved}
    \end{figure}  
    
The picture in $T$ can be more complicated than these panels suggest.  For example, panel 15 might look like Figure \ref{fig:resolved} in $T$.  

\begin{figure}[ht]
%  \centering
  \begin{tikzpicture}[scale=1.5]
    \draw (0.75, -1.5) -- (4.75,-1.5);
    \draw (2.75, -3.5) -- (2.75,0.5);
    \draw  (1,-2.4) ellipse (0.08 and 0.25);
    \draw  (2.3,-2.65) arc (-90:90:0.08 and 0.25);
    \draw [dashed] (2.3,-2.15) arc (90:270:0.08 and 0.25);
    \draw (1,-2.15) to [out=25,in=155] (2.3,-2.15);
    \draw (1,-2.65) to [out=-25,in=205] (2.3,-2.65);
    \draw  (1.3,-2.05) arc (90:-90:0.1 and 0.35);
    \draw  [dashed] (1.3,-2.05) arc (90:270:0.1 and 0.35);
    \draw  (1.85,-2.4) arc (0:-180:0.2 and 0.1); %bottom inner curve
    \draw  (1.8,-2.46) arc (0:180:0.15 and 0.07); %top inner curve
    %Q1
    \draw  (3.2,-0.6) ellipse (0.08 and 0.25); %leftmost ellipse
    \draw (3.2,-0.35) to [out=25,in=155] (4.5,-0.35); %top curve
    \draw (4.5,-0.35) arc (90:-90:0.08 and 0.25); %right ellipse half (solid)
    \draw [dashed] (4.5,-0.35) arc (90:270:0.08 and 0.25); %right ellipse half (dashed)
    \draw (3.2,-0.85) to [out=-25,in=205] (4.5,-0.85); %bottom curve
    \draw  (4.06,-0.6) arc (0:-180:0.2 and 0.1); %bottom inner curve
    \draw  (4.01,-0.66) arc (0:180:0.15 and 0.07); %top inner curve
    \draw (4.22,-0.25) arc (90:-90:0.1 and 0.35); %auxiliary curve right
    \draw [dashed] (4.22,-0.25) arc (90:270:0.1 and 0.35); %auxiliary curve left
    %Q4
    \draw  (3.2,-2.4) ellipse (0.08 and 0.25); %leftmost ellipse
    \draw (3.2,-2.15) to [out=25,in=155] (4.5, -2.15); %top curve
   \draw (4.5,-2.15) arc (90:-90:0.08 and 0.25); %right ellipse half (solid)
    \draw [dashed] (4.5,-2.15) arc (90:270:0.08 and 0.25); %right ellipse half (dashed)
    \draw (3.2,-2.65) to [out=-25,in=205] (4.5,-2.65); %bottom curve
    \draw  (4.06,-2.4) arc (0:-180:0.2 and 0.1); %bottom inner curve
    \draw  (4.01,-2.46) arc (0:180:0.15 and 0.07); %top inner curve
    \draw (3.85, -2) arc (90:-90:0.07 and 0.195); %top auxiliary half (solid)
    \draw [dashed] (3.85, -2) arc (90:270:0.07 and 0.195); %top auxiliary curve (dashed)
    \draw (3.85, -2.5) arc (90:-90:0.05 and 0.155); %bottom auxiliary curve (solid)
    \draw [dashed] (3.85, -2.5) arc (90:270:0.05 and 0.155); %bottom auxiliary curve (dashed)
    %Q2
    \draw  (1,-0.6) ellipse (0.08 and 0.25); %leftmost ellipse
    \draw (1,-0.35) to [out=25,in=155] (2.3,-0.35); %top curve
    \draw (2.3,-0.35) arc (90:-90:0.08 and 0.25); %right ellipse half (solid)
    \draw [dashed] (2.3,-0.35) arc (90:270:0.08 and 0.25); %right ellipse half (dashed)
    \draw (1,-0.85) to [out=-25,in=205] (2.3,-0.85); %bottom curve
    \draw  (1.86,-0.6) arc (0:-180:0.2 and 0.1); %bottom inner curve
    \draw  (1.81,-0.66) arc (0:180:0.15 and 0.07); %top inner curve

    \draw  (1.66,-0.63) ellipse (0.3 and 0.2);
%  \draw [dashed] (1.66,-0.63) ellipse (0.4 and 0.3);
 % \draw (1.26, -0.25) -- (1.26, -0.95);
  \draw   [dashed](1.26, -0.25) arc (90:270:0.1 and 0.35);
     \draw (1.26, -0.25) .. controls (2.35,-0.3) and (2.35,-0.95) .. (1.26, -0.95);
%    \draw [<->,thick] (2.5,-1.25) to (3,-1.75);

%    \node at (0, -4) {Arrow  shows $A, B$ curves that intersect};
  \end{tikzpicture}
  \caption{How panel 15 might be resolved in $T$} \label{fig:resolved}
\end{figure}
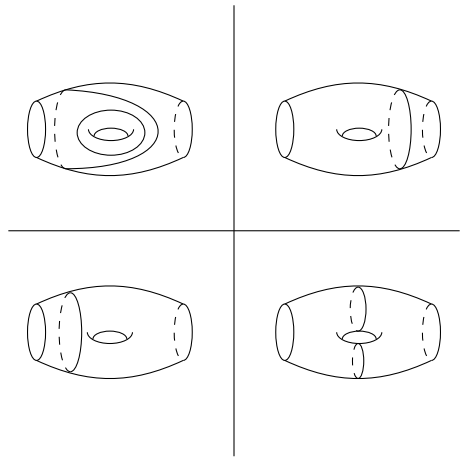
   
\begin{prop} \label{prop:diagonals}
No vertex in the graphic is surrounded by labeling pattern
   $\arraycolsep = 2.0pt
   \begin{array}{c|c}
     A & B \\
     \hline
     B & A
   \end{array}$. 
   \end{prop}
   
\begin{proof}
The simply connected components of $\Ebar - T$ that are shown in Figure \ref{fig:vertexresolved} will each become a disk in some resolution of the tangency points; if the boundary of such a disk is in $\fc_e$, contains no other circles of $\fc_e$ and near its boundary lies in $A$, we will call the the resolve of the disk (which lies entirely in $A$, see Lemma \ref{lemma:hatcher}) an $A$-leaf and the corresponding component of $\Ebar - T$ an $A$-leaf component.  Similarly for a $B$-leaf.   (The terminology is explained in the next section.)    A component of $\Ebar - T$ shown in the diagram is incident either to one of $\rho_{\pm}$ or to both.  

\begin{lemma} \label{lemma:1tangency}  If an $A$-leaf component is incident to only one of $\rho_{\pm}$, then the labeling around the vertex is not  $\arraycolsep = 2.0pt
   \begin{array}{c|c}
     A & B \\
     \hline
     B & A
   \end{array}$
   \end{lemma}
   
   \begin{proof}
Resolve the single tangent point so that the component becomes an $A$-leaf.  Either resolution at the other tangent point (these corresponding to two adjacent quadrants in $\Gamma$) leaves the $A$-leaf intact, so these two adjacent quadrants both get labeled $A$.  
   \end{proof}

\begin{lemma} \label{lemma:2tangency}  At a vertex in $\Gamma$ with surrounding labels 
  $\arraycolsep = 2.0pt
   \begin{array}{c|c}
     A & B \\
     \hline
     B & A
   \end{array}$
   the two $A$ labels cannot both come from $A$-leaf components.  
   \end{lemma}
   
\begin{proof}  Following Lemma \ref{lemma:1tangency} each $A$-label must come from an $A$-leaf component that is incident to both $\rho_{\pm}$.  This eliminates panels 1 through 9.  Moreover, the two $A$-leaf components arise from different resolutions on each tangency point, since they are diagonally opposite.  Only panels 11 and 15 have at least two $2$-vertex components, but in panel 11 they are adjacent and so they can't both lie in $A$.  In panel 15, a resolution of the tangencies that turns an $A$-leaf component into an $A$-leaf, when reversed, only gives disk components that contain points in $B$. 
\end{proof}

This would seem to prove Proposition \ref{prop:diagonals}, until we remember that $A$-labels may arise in another way, as shown in Figure \ref{fig:hidden}. For example in panel 4, the annulus component of $\Ebar - T$ that is shown might lie in $A$, and the interior pair of circles might bound parallel disks in $B$, but when the pair is resolved into a single circle, it is inessential in $\fc$.    Call such a component of $E - T$ in a panel an $A$-annulus.

\begin{lemma} \label{lemma:annulus}  At a vertex in $\Gamma$ with surrounding labels 
  $\arraycolsep = 2.0pt
   \begin{array}{c|c}
     A & B \\
     \hline
     B & A
   \end{array}$,
   neither $A$ label can come from an $A$-annulus.  
   \end{lemma}

\begin{proof} Suppose one of the quadrants gets its $A$-label via an $A$-annulus, as described. Such a component could arise in panels 4, 5, 6, 7, 8 and 9.  In order to have the given labeling 
   the opposite resolution at both $\rho_{\pm}$ should again generate an $A$-label.  The label can't come from the same $A$-annulus since its inner boundary is no longer adjacent to an inessential disk.  Thus the $A$-label must come from an $A$-leaf component and, by observation, each $A$-leaf component is incident to only one of $\rho_{\pm}$.  This contradicts Lemma \ref{lemma:1tangency}
\end{proof}

A final way in which $A$-labels might arise is via `hidden components'.  Remember that the panels only show components of $T \cap E$ that are incident to the tangency points.  Imagine a circle $c$ of $E \cap T$ bounding a disk that contains the pair of components shown in panel 1.  
If both of $\rho_{\pm}$ resolve as in Figure \ref{fig:hidden}, the resulting disk bounded by $c$ could generate a label $A$.  The hidden component here is the `pair of pants' bounded by $c$ and the two components in the panel; it is hidden because $c$ does not appear in the panel.  But it is easy to see that hidden pairs of pants (which could arise in panels 1, 2, and 3) can't possibly give rise to the labeling 
$\arraycolsep = 2.0pt
   \begin{array}{c|c}
     A & B \\
     \hline
     B & A
   \end{array}$.  
   
Hidden annuli require more thought.  Suppose a circle component $c$ of $E \cap T$ cobounds an annulus with a component $X$ from one of the panels.  It may be possible to resolve the tangency points in $X$ so that the end of the annulus at $X$ bounds an inessential circle, so it might in this way be possible for $c$ to generate a label $A$.  By the argument of  Lemma \ref{lemma:1tangency} such a hidden annulus can be part of a labeling
 $\begin{array}{c|c}
     A & B \\
     \hline
     B & A
   \end{array}$
   only if the end at $X$ is incident to both tangency points $\rho_{\pm}$.  This immediately rules out panels 1 through 9 as well as 11 and 14.  The end at $X$ must also have the property that the opposite resolution at both $\rho_{\pm}$ will still give rise to an $A$-disk, and the new $A$-disk must be incident to both $\rho_{\pm}$, by Lemma \ref{lemma:1tangency}.  This eliminates panels 12 and 13.  Panel 10 won't work: each leaf component shown has points in $B$, since the hidden annulus lies in $A$. Finally, these requirements can be fulfilled in panel 15 only if the middle sector lies in $A$ and the two other sectors lie in $B$, and are inessential in $\fc$.  But in that case, there could be no label $B$ in any quadrant.  

A technical note: our labeling convention assigns the label $A$ also if one of the regions in $\Ebar - T$ is a collar of the boundary in an $A$-disk.  The argument in this case is identical to that given above for the case in which there is a hidden circle that completely surrounds the figure in each panel.  
\end{proof}

\section{From weakly aligned to aligned} \label{sect:weakreduce}

In \cite{CG} Casson-Gordon introduced the notion of a weakly reducible Heegaard splitting, rejuvenating Heegaard theory.  They showed that if there are disjoint essential disks in $A$ and $B$, then simultaneous compression on a maximal family of such disjoint disks in $A$ and $B$ will produce either a reducing sphere for $T$ or an incompressible surface or both.  In considering uniqueness, as we are doing here, the choice of a `maximal family' is problematic, since such a family is far from unique.   In this section we avoid this problem of choice, by deriving from the entire pattern of circles $T \cap E$ in $E$ a recipe to move from what we will call a weak alignment of $T$ (in analogy to `weak reduction') to a full alignment of $T$, in a series of steps that is well-defined up to congruence.  

Suppose $E$ is a disk/sphere set in $M = A \cup_T B$.  $E$ will be called in {\em weak alignment} with $T$ if, among the components of $E - T$, there are both $A$-leaves and $B$-leaves.  Continuing under Assumption \ref{ass:induct}, we will describe a natural algorithm that transforms a weak alignment of $E$ into a full alignment, an algorithm that is well-defined up to congruence.  

Denote by $\Da$ (resp $\Db$) the collection of all $A$-leaves (resp $B$-leaves) in $E$ coming from $E \cap T$.  Figure \ref{fig:Eview} illustrates the idea in a disk component of $E$ whose boundary lies in $\bdd_- B$.  

\begin{figure}[ht!]
\labellist
\small\hair 2pt
\pinlabel  $\Da$ at 110 35
%\pinlabel  $\gamma$ at 155 133
\pinlabel  $\Db$ at 175 150
\endlabellist
    \centering
    \includegraphics[scale=0.6]{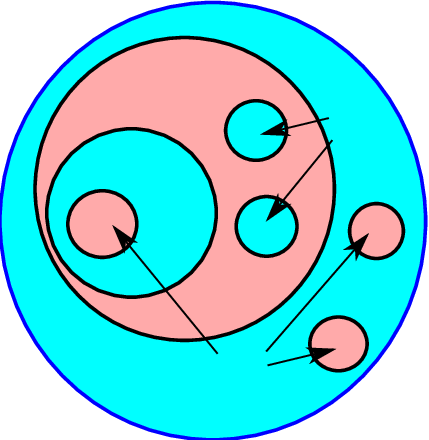}
\caption{The view in $E$} 
 \label{fig:Eview}
    \end{figure}

Consider the surface $T_{c(ompressed)} \subset M$ obtained by compressing $T$ along $\mathcal{D}_A \cup \mathcal{D}_B $.  $T_c$ divides $M$ into two (each possibly disconnected) 3-manifolds $M_A$ and $M_B$.  Imagine thickening $T_c$ by expanding it into a bi-collar as shown in Figure \ref{fig:MAMB}.  This would induce Heegaard splitting surfaces $T_A \subset M_A$, obtained from the original $T$ by compressing only along $\mathcal{D}_A$ and then pushing  towards the $A$-side of the bicollar.  The symmetric construction gives a Heegaard splitting surface $T_B$ in $M_B$.
%T_B$, so that $$M = M_A \cup_{T_c} M_B$.  

\begin{figure}[ht!]
\labellist
\small\hair 2pt
\pinlabel  $\Da$ at 120 115
\pinlabel  $T$ at 135 200
\pinlabel  $\Db$ at 70 140
\pinlabel  $A$ at 30 115
\pinlabel  $B$ at 150 115
\pinlabel  $M_A$ at 200 140
\pinlabel  $T_A$ at 200 215
\pinlabel  $M_B$ at 285 110
\pinlabel  $T_B$ at 285 200
\pinlabel  $T_c$ at 245 200
%\pinlabel  $M_A$ at 70 140
\endlabellist
    \centering
    \includegraphics[scale=0.6]{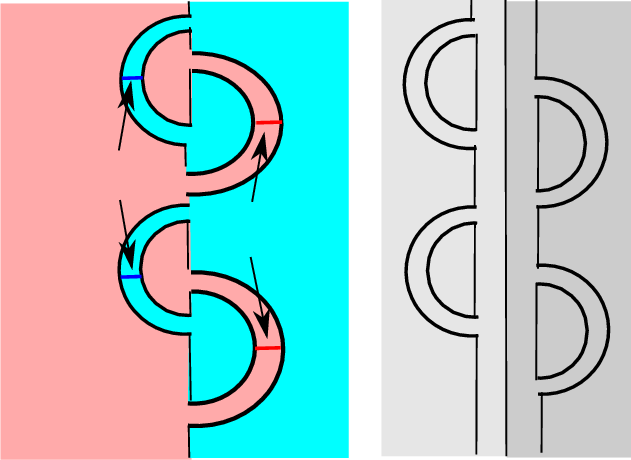}
\caption{The view in $M$, and a mental image} 
 \label{fig:MAMB}
    \end{figure}
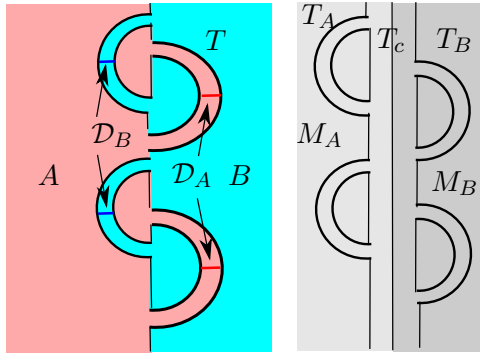

Both $(M_A, T_A)$ and $(M_B, T_B)$ have lower complexity than $(M, T)$, so our inductive Assumption \ref{ass:induct} applies.  In particular, given any disk/sphere set in $M_A$, the surface $T_A$ can be isotoped, uniquely up to congruence, so that it aligns with the disk/sphere set (and similarly for $(M_B, T_B)$).  Such an isotopy of $T_A$ can be described \cite{Sc} as a sequence of handle-slides of and over the handles whose cocores are the $\Db$ disks.  But these same handle-slides could have been done on and over the handles as they actually lie on $T_c$, avoiding (by general position) the attaching disks for the handles on the other side, those with cocores the disks $\Da$.  In thinking of this as an isotopy of the original Heegaard surface $T$, the exact trajectory which the handle-slides follow across $T_c$ in order to avoid the disks $\Db$ is, for our purposes, unimportant: one choice can be moved to another by eyeglass twists of $T$.  The symmetric argument applies to $M_B$.  The upshot is:

\begin{prop} \label{prop:handleslide}
Suppose $E_A \subset M_A$ and $E_B \subset M_B$ are disjoint disk/sphere sets. %whose boundaries lie on $T_c \subset M$.   
Then there is an isotopy of $T$, keeping $T_c$ set-wise fixed, to a position in which the boundary of each disk in $E_A \cup E_B$ remains unchanged in $T_c$ and $T$ is disjoint from the interior of each disk in $E_A \cup E_B$.  The isotopy of $T$ is well-defined up to congruence.  \qed
\end{prop}

Consider a component $P$ of $E - \fc_e$ which is next to innermost, i.e. all but one of its boundary components is an innermost circle in $\fc_e$, so each of these components of $\bdd P$ bounds a disk in $\Da$ (or each bounds a disk in $\Db$), the resolve of the disk it bounds in $E$.  Then the exceptional boundary component $\bdd_0 P$ lies in $T_c$ and bounds a disk $D_P$ in $M_B$ (or $M_A$), through which the $1$-handles dual to $\Da$ (or $\Db$) may pass.  
\medskip

The algorithm is then:  
\begin{enumerate}
\item  Apply Proposition \ref{prop:handleslide} to the collection of all such components $P$ of $E - T$, isotoping $T$ without changing $T_c$ so that afterwards the interior of each disk $D_P$ is disjoint from $T$.
\item Add each such disk $D_p$ to $\Db$ or $\Da$ as appropriate, compressing $T_c$ to $T'_c$
\item Repeat the process until at least one component $\overline{E}$ of $E$ is aligned with $T$ (as explained below).
\item The output is the alignment generated by $\overline{E}$.
\end{enumerate}

It will be important for its application that the algorithm is robust: a minimal change in input information will result in the same output.  To understand more fully how the algorithm operates, we can describe it schematically.

 The pattern in $E$ of $\fc_e$ defines a tree in each component $\Ebar$ of $E$, with a vertex for each component of $\Ebar - \fc_e$ and an edge connecting two such components if there is a single circle of $\fc_e$ between them.  The tree has a natural base or root when $\Ebar$ is a disk, namely the component of $\Ebar - \fc_e$ containing the boundary.  Let $Y$ denote the forest that is the whole collection of trees.  The innermost disks of $E - \fc_e$ can be thought of as leaves in the forest $Y$. One measure of the complexity of each tree is the diameter of the tree, when $\Ebar$ is a sphere, or the height of the tree when $\Ebar$ is a disk.  (Tree height is the edge-distance from the root of the tree to the most distant leaf.  See Figure \ref{fig:Tree1}). 

\begin{figure}[ht!]
\labellist
\small\hair 2pt
\pinlabel  $0$ at 550 60
\pinlabel  $4$ at 315 180
\pinlabel  $2$ at 360 220
%\pinlabel  $M_A$ at 70 140
\endlabellist
    \centering
    \includegraphics[scale=0.5]{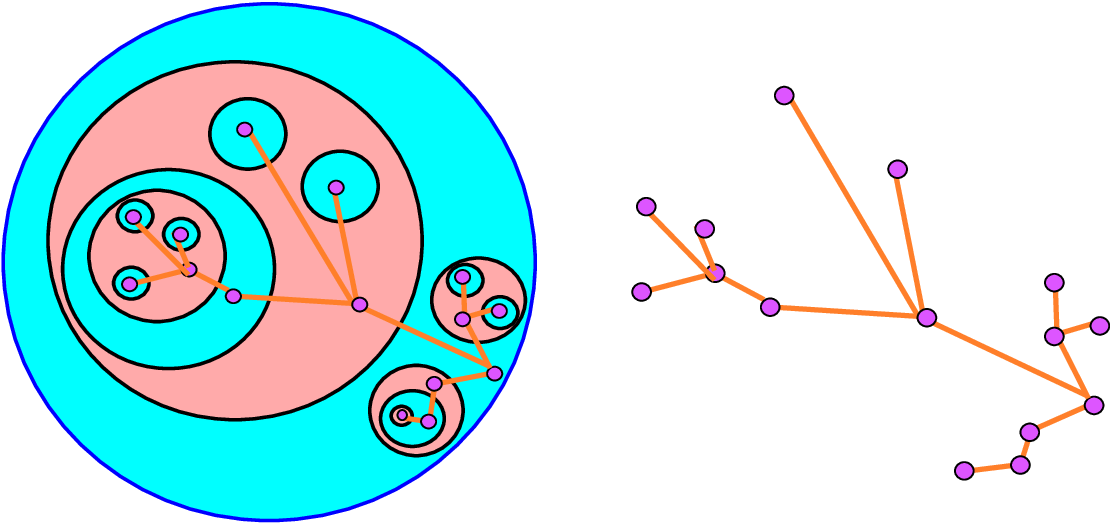}
\caption{Tree height is 4} 
 \label{fig:Tree1}
    \end{figure}

The $B$-leaves of $Y$ correspond to $\mathcal{D}_B$, and $A$-leaves to $\mathcal{D}_A$.  
The {\em branch-structure} $Y_c$ of $Y$ is obtained from $Y$ by removing all leaves; alternatively, it is the forest determined by the circles $T_c \cap E$ in $E$.  

\begin{figure}[ht!]
\labellist
\small\hair 2pt
\pinlabel  $Y_c$ at 400 60
\endlabellist
    \centering
    \includegraphics[scale=0.5]{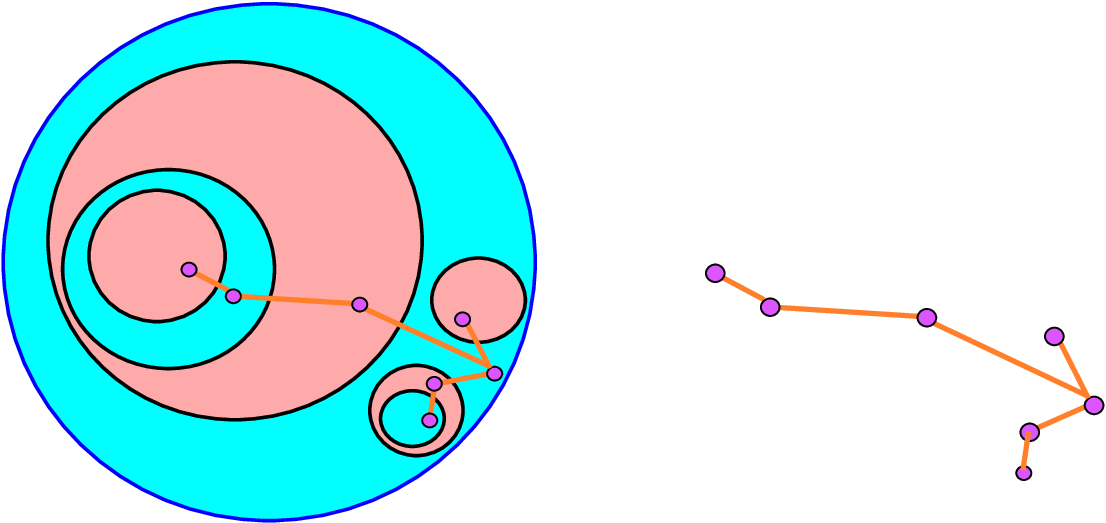}
\caption{Branch structure from $T_c \cap E$} 
 \label{fig:Branch}
    \end{figure}
    
    The leaves of the branch structure correspond to the ``second-innermost" circles in the algorithm described above or, in terms of the original forest, they are the vertices that become leaves when the original leaves are removed. %, often outermost forks in the original forest.   
Applying the algorithm described above replaces the original leaves with these new leaves.  Since we have no control over how the $1$-handles of $T_A \subset M_A$ and $T_B \subset M_B$ intersect the non-disk components of $E - T_c$, leaves may also sprout from every other vertex in $Y_c$.  But one iteration of the algorithm described will decrease the height (or diameter) of each tree.  This is shown schematically in Figure \ref{fig:Step2b}, where new leaves sprout in non-disk components of $E - T_c$.    

\begin{figure}[ht!]
\labellist
\small\hair 2pt
%\pinlabel  {Some } at 350 80
\pinlabel  {New} at 350 80
\pinlabel  {sprouting} at 350 60
\endlabellist
    \centering
    \includegraphics[scale=0.6]{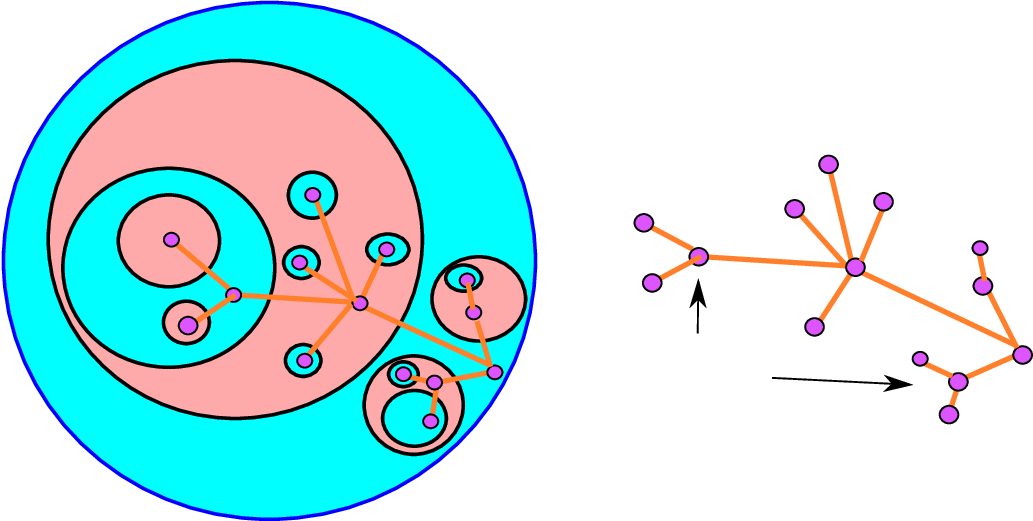}
\caption{De/refoliation of $B$-leaves; height is now 3} 
 \label{fig:Step2b}
    \end{figure}

 Since the complexity of $(M_A, T_A)$ is less than the complexity of $(M, A)$ inductive Assumption \ref{ass:induct} says that the new $A$-leaves of $Y'$ (the ones corresponding to the leaves of the branch structure) are well-defined in $T_A$ up to congruence, so they are similarly well-defined in $T$.  Add them to $\mathcal{D}_A$, compressing $T_c$ into $A$, typically $\bdd$-reducing $T_A \subset M_A$. See Figure \ref{fig:MakeY2}. Call the augmented collection $\mathcal{D'}_A$.

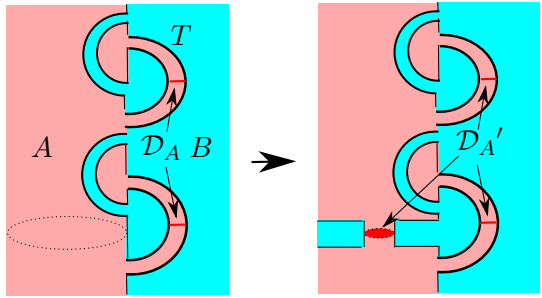
\begin{figure}[ht!]
\labellist
\small\hair 2pt
\pinlabel  $\Da$ at 120 115
\pinlabel  $T$ at 135 200
\pinlabel  $\Da'$ at 360 120
\pinlabel  $A$ at 30 115
\pinlabel  $B$ at 150 115
%\pinlabel  $A$ at 260 115
%\pinlabel  $B$ at 385 115
%\pinlabel  $M_A$ at 70 140
\endlabellist

    \centering
    \includegraphics[scale=0.5]{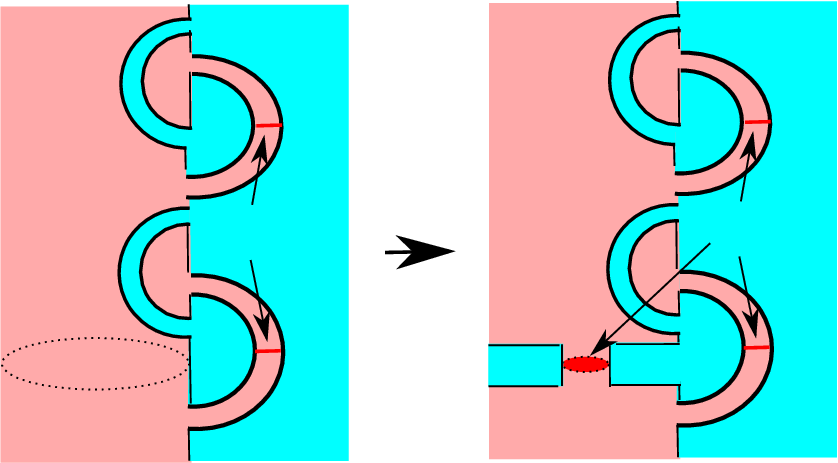}
\caption{New $A$-leaves added to $\mathcal{D}_A$ } 
 \label{fig:MakeY2}
    \end{figure}
     
A symmetric argument applies in $M_B$, resulting in new surfaces $T'_A, T'_B$ and $T'_c$,  the latter dividing $M$ now into $M'_A, M'_B$. 

Continue with the algorithm until the height or diameter in some component $\overline{E}$ is $1$. We note in detail the last step:  $T$ now divides $\overline{E}$ into a planar surface in $B$, say, (which is incident to  $\bdd \overline{E}$ if $\overline{E}$ is a disk) and a collection of disks in $A$, all of $\overline{E}$ now lying in the Heegaard split proper submanifold $M_B$ of $M$ with lower complexity than $(M, T)$.  Once again apply the Strong Haken Theorem \cite{Sc} together with the inductive hypothesis on this lower complexity splitting to isotope $T$ so that $\overline{E}$ is aligned, uniquely up to congruence.  This completes the algorithm.

\section{Appendages and inherited alignments} \label{sect:appendage}

Section \ref{sect:weakreduce} described an algorithm which proceeds by well-defined iteration from a weak alignment of $T$ with a disk/sphere set $E$ into a full alignment, at each step increasing the number of disjoint weakly reducing disks.  We picture those disks originally corresponding to outermost leaves moved via handleslides to make way for new disks corresponding to outermost leaves of the branch structure.  Suppose $E'$ is a disk/sphere set contained in the weakly aligned $E$ and $E'$ is itself weakly aligned.  In this section we show that the algorithm applied to $E'$ gives an alignment of $E'$ that is congruent to the alignment of $E'$ inherited from the alignment of $E$.  

We begin with this easy corollary of our inductive assumption:

\begin{cor}\label{cor:appendage}  Let $E_0, E_1$ be properly isotopic disk/sphere sets in $M$ that are congruently aligned with $T$.  Suppose that a disk $(D, \bdd D) \subset (M, \bdd M)$ is disjoint from $E_0$ and $E_1$ and is also aligned with $T$.  Then $E_0$ and $E_1$ are congruent using only eyeglass twists and bubble moves that are disjoint from $D$.
\end{cor}

\begin{proof} The Corollary follows from Assumption \ref{ass:induct}, which allows us to apply Theorem \ref{thm:main} to the lower-genus Heegaard split manifold $(M', T')$ obtained from $(M, T)$ by surgery along $D$. 
\end{proof}

%*****************
%
%\begin{defin} \label{defin:converge}  Two weak alignments of $T$ with $E$ {\em converge} if the algorithm results in congruent full alignments.
%\end{defin}
%
%The term is meant to convey that, after perhaps some iterations in the algorithm, the two weak alignments may become indistinguishable, even before they each become fully aligned.  
%
%An important example of convergent weak alignments begins with this easy corollary of our inductive assumption:
%
%\begin{cor}\label{cor:appendage}  With hypotheses as in Theorem \ref{thm:main},  suppose that a disk $(D, \bdd D) \subset (M, \bdd M)$ is disjoint from the two disk/sphere sets $ E_i$ and is aligned with $T$.  Then $E_0$ and $E_1$ are congruent using only eyeglass twists and bubble moves that are disjoint from $D$.
%\end{cor}
%
%\begin{proof} The Corollary follows from Assumption \ref{ass:induct}, which allows us to apply Theorem \ref{thm:main} to the lower-genus Heegaard split manifold $(M', T')$ obtained from $(M, T)$ by surgery along $D$. 
%\end{proof} 

Suppose $E$ is a disk/sphere set weakly aligned with $T$, and $\Da \subset M_A$ and $\Db \subset M_B$ are the leaves, as described above.  

\begin{prop} \label{prop:appendage}
Suppose $(D, \bdd D) \subset (M_B, \bdd M_B)$ (resp  $(M_A, \bdd M_A))$ is a properly embedded essential disk that is disjoint from $E$.  Then the alignment of $T$ with $E$ given by the algorithm is unaffected (up to congruence) by adding $D$ to $\Db$ (resp $\Da$) at the start.
\end{prop}

\begin{proof}  We will show that in both $M_A$ and $M_B$ the algorithm is unaffected by the addition of $D$ to $\Db$.  (Of course if $D$ is parallel in $B$ to an element in $\Db$ there is nothing to show.)

The submanifold $M_B$ is obtained by compressing $B$ along the disks $\Db$; we can think of $D$ as an aligned disk in $(M_B, T_B)$.  Apply Corollary \ref{cor:appendage} to $D$ in $M_B$, observing then that the algorithm can proceed, whenever it involves handle-slides, bubble moves and eyeglass twists in $M_B$, just as it would if we had first compressed $M_B$ along $D$.  

In $M_A$ the addition of $D$ to $\Db$ changes the status of $\bdd D$ from just being part of the boundary of $M_A$ to being the belt curve of a $1$-handle in the splitting of $M_A$ by $T_A$.  This is a profound change, since the original algorithm may require passing $1$-handles past $\bdd D$.  But the passing of a $1$-handle past $\bdd D$ can be mimicked, after adding $D$ to $\Db$, by passing the $1$-handle over the new $1$-handle dual to $D$.  Moreover, since $D$ is disjoint from $E$, the algorithm never requires the new $1$-handle to be slid at all.  And so the algorithm can proceed step after step, never moving the new $1$-handle, until the alignment is achieved.
\end{proof}

Because of its inactivity in the proof:
\begin{defin} A disk $D$ as in Proposition \ref{prop:appendage} is called an {\em appendage disk}.
\end{defin}

\begin{lemma} \label{lemma:redundant} Suppose $E'$ is a disk/sphere set contained in the weakly aligned disk/sphere set $E$ and $E'$ is itself weakly aligned.  Then the alignment of $E'$ provided by the algorithm applied to $E'$ is congruent to the alignment inherited by $E'$ from the algorithm applied to $E$. 
%(That is, the two weak alignments converge on $E'$.)
\end{lemma}

\begin{proof} The case is in which $E$ consists of only two components $E = \Ebar_{\pm}$, with $E' = \Ebar_+$ is definitive; the general case is no harder.  At the beginning of the algorithm on $E$, the submanifolds $M_A$ and $M_B$ are defined by compressing $T$ along the disk components of $E - \fc_e$.  Now remove $\Ebar_-$ and note that the algorithm applied to $\Ebar_+$ would call for compressing only along the disk components of $\Ebar_+ - \fc_e$.  But the outcome of that algorithm is unaffected by further compressing by disk components on $\Ebar_- - \fc_e$, by Proposition \ref{prop:appendage}.  So, at the initial stage, there is no difference between the eventual alignments.  Just continue in this manner, using how the algorithm behaves on the `virtual' component $\Ebar_-$ to present extra disks to be included as appendages (under Proposition \ref{prop:appendage}) as the algorithm is applied to $\Ebar_+$ alone.  

Eventually the parallel algorithms stop, when one of $\Ebar_{\pm}$ is aligned.  If it stops because $\Ebar_{+}$ is aligned, then we have shown that the alignments from $E$ and from $\Ebar_+$ result in the same alignment of $\Ebar_+$, as required.  If it stops because $\Ebar_{-}$ is aligned, then the algorithm for $E$ declares that the alignment consists of $\Ebar_{-}$, together with {\em any} alignment for $\Ebar_{+}$ in the manifold obtained by surgering %reducing (or $\bdd$-reducing)
 $(M, T)$ along $\Ebar_{-}$.  An example of such an alignment is given by the output of the algorithm further played out on $\Ebar_+$.
\end{proof}

Suppose $E$ is a weakly aligned disk/sphere set containing two components $\Ebar_+, \Ebar_-$ that are parallel to each other in $M$, with no component of $E$ between them.  Let $E_+ = E - \Ebar_- \subset E$ and $E_- = E - \Ebar_+ \subset E$.  Note that $E_+, E_-$ are themselves isotopic disk/sphere sets in $M$.  Use the algorithm to align $E$.

\begin{prop} \label{prop:parallel} The alignments of $E_{\pm}$ inherited from the alignment of $E$ are congruent.
\end{prop}

Note that we are not assuming either of $E_{\pm}$ is originally weakly aligned, though by Lemma \ref{lemma:redundant} if either is, the alignment given by the algorithm is congruent to that inherited from the alignment of $E$.

\begin{proof}  After $E$ is aligned, consider the product region between $\Ebar_{\pm}$ (either $B^2 \times I$ or $S^2 \times I$). Since $E$ is aligned, $T$ may intersect each end of this product region in at most one circle. %We may as well assume there are no other components of $E$ in this product region.   
If $T$ is disjoint from both ends, then it is disjoint from the product region, so the inherited alignments on the components  $\Ebar_{\pm}$ are isotopic. 

If $T$ intersects only one end, say $\Ebar_+$ then it follows from Waldhausen's theorem \cite{Wa} that $T$ intersects the product region in a collection of bubbles; the inherited alignments on the components  $\Ebar_{\pm}$ then differ by bubble moves, so they are congruent.  If $T$ intersects both ends, a similar argument shows that, after some bubble moves, $T$ will intersect the product region in a simple annulus, so the alignments are again congruent.  
\end{proof}

\section{Return to the graphic} \label{sect:graphicreturn}

%Suppose an edge in the graphic lies between a region labeled $A$ and a region labeled $B$.  The edge indicates a saddle tangency of $E$ with $T$.  Let $\overline{E}$ be the component of $E$ that contains the saddle tangency.  Let $\Ebar_A, \Ebar_B$ be slight push-offs of $\Ebar$ corresponding to the regions labeled $A$ and $B$ respectively, and $E_{\pm}$ be the weakly aligned disk/sphere set obtained from $E$ by deleting $\Ebar$ and replacing with $\Ebar_A \cup \Ebar_B$.  The algorithm applied to $E_{\pm}$ aligns $E_{\pm}$ and, by Proposition \ref{prop:parallel}, the inherited alignments on the two positionings of $E$ corresponding to the adjacent regions are congruent.  
%
%\begin{defin} This congruence class of alignments on $E$ is called the {\em border edge alignment}.  \end{defin}
%
%Note again that  by Lemma \ref{lemma:redundant} the border edge alignment is congruent to that given by the algorithm applied in the adjacent regions, if those adjacent regions correspond to a weak alignment.  

Return now to the proof of Theorem \ref{thm:main} by examining the graphic more closely, inspired by \cite[Subsection 4.5]{FS} and adopting similar conventions.  An edge in the graphic that lies between a region labeled $A$ and a region labeled $B$ will be called a {\em border edge}.  Following Section \ref{sect:forbidden}, any vertex in the graphic that is incident to a border edge is incident to exactly two border edges (or to the boundary of the graphic).  Thus the collection of border edges constitute a properly embedded $1$-manifold in the graphic, one that separates $A$ regions from $B$ regions.  

We have shown earlier that three sides of the graphic ($s = 0, 1$ and $t = 0$) are adjacent to $A$-regions.  Since the union of the three sides is connected, there is a single component $\mathcal{A}$ of the complement of the border edges that contains all three sides in its boundary. $\mathcal{A}$ consists entirely of regions labeled $A$.  (See Figure \ref{fig:graphicsketch}.)

We focus on the $1$-manifold component $C$ of $\bdd \mathcal{A}$  that contains the three sides $s = 0, 1$ and $t = 0$.  If $C$ contains the fourth side $t=1$ then per Proposition \ref{prop:side1} we are done, so our interest will focus on the arc in $C$ whose ends are at the corners $s \in \{0, 1\}, t = 1$ of the graphic, or more specifically, the border arcs that $C$ contains.    (See Figure \ref{fig:graphicsketch}).  

%
%\begin{figure}[ht] \label{fig:favorA}
%  \centering
%  \begin{tikzpicture}[scale=0.7]
%    \draw (0, 0) -- (4, 4);
%    \draw (0, 4) -- (4, 0);
%    \node at (2, 0.75) {\large A};
%    \node at (0.5, 2) {\large B};
%    \node at (3.5, 2) {\large B};
%    \node at (2, 3.25) {\large A};
%
%    \node at (6, 2) {\Huge $\longrightarrow$};
%    \node at (5.93, 2.35) {\small render};
%
%    \draw (8, 0) -- (9.4, 1.4);
%    \draw (8, 4) -- (9.4, 2.6);
%    \draw (9.4, 1.4) to [out = 45, in = 315] (9.4, 2.6);
%    \draw (12, 0) -- (10.6, 1.4);
%    \draw (12, 4) -- (10.6, 2.6);
%    \draw (10.6, 1.4) to [out = 135, in = 225] (10.6, 2.6);
%    \node at (10, 0.75) {\large A};
%    \node at (8.5, 2) {\large B};
%    \node at (11.5, 2) {\large B};
%    \node at (10, 3.25) {\large A};
%  \end{tikzpicture}
%  \caption{}
%\end{figure}
%
%\underln{Artistic convention}: If labels $A$ and $B$ alternate around a saddle wall double point (see Figure \ref{fig:favorA}) render the boundary between $A$-regions and $B$-regions to be a 1-manifold favoring $A$. 

%The resulting 1-manifold separates $A$ regions from $B$ regions;

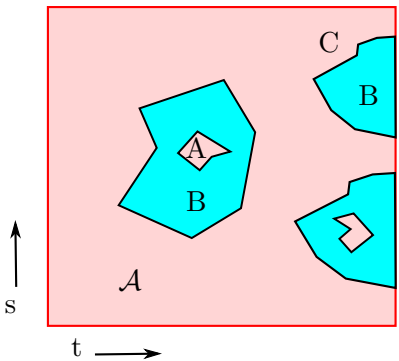
\begin{figure}[ht!]
\labellist
\small\hair 2pt
\pinlabel  {s} at 45 20
\pinlabel  {t} at 70 05
\pinlabel  {$\mathcal{A}$} at 90 30
\pinlabel  {B} at 180 100
\pinlabel  {B} at 115 60
\pinlabel  {A} at 115 80
\pinlabel  {C} at 165 120
\endlabellist
    \centering
    \includegraphics[scale=1.0]{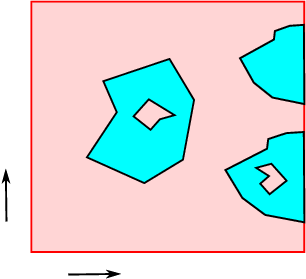}
\caption{Graphic labels: red = A, blue = B} 
 \label{fig:graphicsketch}
    \end{figure}
    
A border edge indicates a saddle tangency of $E$ with $T$.  Let $\overline{E}$ be the component of $E$ that contains the saddle tangency.  Let $\Ebar_A, \Ebar_B$ be slight push-offs of $\Ebar$ corresponding to the regions labeled $A$ and $B$ respectively, and $E_{\pm}$ be the weakly aligned disk/sphere set obtained from $E$ by deleting $\Ebar$ and replacing with $\Ebar_A \cup \Ebar_B$.  The algorithm applied to $E_{\pm}$ aligns $E_{\pm}$ and, by Proposition \ref{prop:parallel}, the inherited alignments on the two positionings of $E$ corresponding to the adjacent regions are congruent.  

\begin{defin} This congruence class of alignments on $E$ is called the {\em border edge alignment}.  \end{defin}

Note again that  by Lemma \ref{lemma:redundant} the border edge alignment is congruent to that given by the algorithm applied in the adjacent regions, if those adjacent regions correspond to a weak alignment.

It follows that the lowest border edge (i. e. minimal $s$) on the lowest border arc in $C$ has border edge alignment congruent to $E_0$ and the highest border edge on the highest border arc of $C$ has border edge alignment congruent to $E_1$.  If we can show that the border edge alignments given by successive border edges in $C$ are always congruent, then we will have shown that the solutions $E_0$ and $E_1$ are congruent, as required.  So we examine how passing through a vertex of the graphic that lies on a border arc affects the border edge alignments of the incident edges.  We will show the following, from which Theorem \ref{thm:main} then follows.

\begin{prop} \label{prop:vertex}
At any vertex in a border arc, the border edge alignments given by the incident border edges are congruent.
\end{prop}

\begin{proof}  There is an important feature distinguishing between the two diagonals in a labeling diagram around a vertex in $\Gamma$.  
%$\arraycolsep = 2.0pt
%   \begin{array}{c|c}
%     \bullet &   \\
%     \hline
%       & \bullet
%   \end{array}$ versus
%  $\arraycolsep = 2.0pt
%   \begin{array}{c|c}
%      & \bullet \\
%     \hline
%     \bullet &
%   \end{array}$, a feature which we first describe. 
 Put $T$ in the position determined by the vertex of $\Gamma$, so that $E = E_s$ and $T = T_t$ are tangent at two points $\rho = \rho_{\pm}$.  We assume that $\rho_{\pm}$ lie on the same component $\Ebar$ of $E$; if they lie on different components of $E$ the proof is easier.   

Let $\Ebar_{\pm}$ be slight push-offs of $\Ebar$ to each of its sides.  Then the disks $\Ebar_{\pm}$ correspond to positionings of $\Ebar$ that lie in diagonally opposite quadrants of the graphic, since in moving from one to the other, the resolution of each of the tangencies at $\rho_{\pm}$ is changed.  The curves of $\Ebar_{+} \cap T$ and $\Ebar_{-} \cap T$ are visibly disjoint  in $T$, since the disks $\Ebar_{\pm}$ are disjoint in $M$.  Call this the {\em safe} diagonal.  (The other diagonal was called the {\em dangerous diagonal} in \cite{FS}. In Figure \ref{fig:resolved} the antidiagonal is safe and the main diagonal is dangerous.)
\bigskip

The argument will be symmetric in $A$ and $B$ and also indifferent to symmetries of the quadrants about the vertex, so, following Proposition \ref{prop:diagonals}, there are just two cases to consider,  corresponding to the labelings:
$\arraycolsep = 2.0pt
   \begin{array}{c|c}
     A & A \\
     \hline
     A & B
   \end{array}$ and 
   $\arraycolsep = 2.0pt 
   \begin{array}{c|c}
     A & A \\
     \hline
     B & B
   \end{array}$.

{\bf Case 1:} The labelings around the vertex are 
$\arraycolsep = 2.0pt
   \begin{array}{c|c}
     A & A \\
     \hline
     A & B
   \end{array}$; 
   and the antidiagonal 
   $\arraycolsep = 2.0pt
   \begin{array}{c|c}
      & \bullet \\
     \hline
     \bullet &
   \end{array}$ is safe.

   In this case, replace the component $\Ebar$ in $E$ by three parallel components, namely, the two components $\Ebar_{\pm}$ representing the antidiagonal, and a component $\Ebar_B$ representing the quadrant labeled $B$.  Call the resulting weakly aligned disk/sphere set $E^+$.  

Deleting exactly $\Ebar_+$ from $E^+$ gives the border edge alignment for one boundary edge and deleting exactly $\Ebar_-$ gives the border edge alignment for the other boundary edge. Now apply Lemma  \ref{lemma:redundant}: both of these alignments are inherited from that of $E^+$ and so they are congruent.

{\bf Case 2:} The labelings around the vertex are 
$\arraycolsep = 2.0pt
   \begin{array}{c|c}
     A & A \\
     \hline
     A & B
   \end{array}$; 
   and the main diagonal 
   $\arraycolsep = 2.0pt
   \begin{array}{c|c}
     \bullet &   \\
     \hline
       & \bullet
   \end{array}$  is safe.
%   \bigskip
   
In a similar fashion, replace $\Ebar$ in $E$ by three parallel components:  $\Ebar_+$ representing the upper left quadrant, $\Ebar_-$ representing the lower right quadrant and a component $\Ebar^{12}$ representing the upper right quadrant.  Call the resulting weakly aligned disk/sphere set $E^{12}$.  
 
Deleting exactly $\Ebar_+$ from $E^{12}$ gives the boundary edge alignment for the right boundary edge; deleting exactly $\Ebar^{12}$ gives a weakly aligned disk/sphere set we call here the {\em diagonal set}.  By  Lemma  \ref{lemma:redundant} the alignment on the diagonal set given by the algorithm coincides with that inherited from the alignment of $E^{12}$, as does the boundary edge alignment.  By Proposition \ref{prop:parallel} the two alignments are congruent. 

Now repeat the argument using the disk/sphere set $E^{21}$ obtained by replacing $\Ebar^{12}$ with a component $\Ebar^{21}$ representing the lower left quadrant. The argument shows that the alignment coming from the diagonal set is congruent to the border edge alignment for the lower boundary edge.  Therefore the border edge alignments representing the two boundary edges are congruent.

{\bf Case 3:}  The labelings around the vertex are
$\arraycolsep = 2.0pt 
   \begin{array}{c|c}
     A & A \\
     \hline
     B & B
   \end{array}$.
 In this case we may as well assume the main diagonal is safe.  Then a minor variant of the argument for Case 2 suffices.
 \end{proof}

\end{document}